\documentclass[12pt, a4paper]{amsart}
\usepackage{amsmath}
\usepackage{geometry,amsthm,graphics,tabularx,amssymb,shapepar}
\usepackage{amscd}
\usepackage{mathrsfs}
\usepackage[all]{xypic}

\newcommand{\BH}{{\mathbb {H}}}

\newcommand{\BQ}{{\mathbb {Q}}}

\newcommand{\CA}{{\mathcal {A}}}

\newcommand{\Aut}{{\mathrm{Aut}}}

\newcommand{\GL}{{\mathrm{GL}}}

\newcommand{\Hom}{{\mathrm{Hom}}}

\newcommand{\Ind}{{\mathrm{Ind}}}

\newcommand{\Lie}{{\mathrm{Lie}}}

\newcommand{\rank}{{\mathrm{rank}}}

\renewcommand{\Re}{{\mathrm{Re}}}

\newcommand{\Span}{{\mathrm{Span}}}

\newcommand{\con}{\textit{C}}

\newcommand{\od}{\operatorname{d}}

\newcommand{\oH}{\operatorname{H}}

\newcommand{\oU}{\operatorname{U}}
\newcommand{\oZ}{\operatorname{Z}}

\newcommand{\g}{\mathfrak g}

\renewcommand{\k}{\mathfrak k}
\newcommand{\h}{\mathfrak h}
\newcommand{\p}{\mathfrak p}
\newcommand{\q}{\mathfrak q}
\renewcommand{\a}{\mathfrak a}

\renewcommand{\c}{\mathfrak c}
\newcommand{\n}{\mathfrak n}
\renewcommand{\u}{\mathfrak u}

\newcommand{\e}{\mathfrak e}
\newcommand{\f}{\mathfrak f}
\renewcommand{\l}{\mathfrak l}
\renewcommand{\t}{\mathfrak t}
\newcommand{\s}{\mathfrak s}

\renewcommand{\o}{\mathfrak o}
\newcommand{\m}{\mathfrak m}
\newcommand{\z}{\mathfrak z}

\renewcommand{\sl}{\mathfrak s \mathfrak l}

\newcommand{\Z}{\mathbb{Z}}
\newcommand{\C}{\mathbb{C}}
\newcommand{\R}{\mathbb R}

\newcommand{\A}{\mathbb{A}}

\newcommand{\abs}[1]{\lvert#1\rvert}

\newcommand{\la}{\langle}
\newcommand{\ra}{\rangle}

\newcommand{\be}{\begin {equation}}
\newcommand{\ee}{\end {equation}}
\newcommand{\bee}{\begin {equation*}}
\newcommand{\eee}{\end {equation*}}

\theoremstyle{Theorem}

\theoremstyle{Theorem}

\theoremstyle{Theorem}
\newtheorem{lem}{Lemma}[section]

\newtheorem{thml}[lem]{Theorem}

\newtheorem{prpl}[lem]{Proposition}

\theoremstyle{Theorem}
\newtheorem{prp}{Proposition}[section]

\newtheorem{lemp}[prp]{Lemma}

\theoremstyle{remark}

\newtheorem{remarkl}[lem]{Remark}
\newtheorem{remarksl}[lem]{Remarks}

\theoremstyle{remark}
\newtheorem{examplel}[lem]{Example}

\newtheorem{examplesl}[lem]{Examples}

\theoremstyle{Definition}

\newtheorem{dfnp}[prp]{Definition}
\newtheorem{dfnl}[lem]{Definition}

\begin{document}

\title{Low degree cohomologies of congruence groups}

\author{Jian-Shu Li}
 
\address{Department of Mathematics, Hong Kong University of Science and Technology,
Clear Water Bay, Kowloon, 
 Hong Kong}
  \email{matom@ust.hk}

\author{Binyong Sun}

\address{Academy of Mathematics and Systems Science, Chinese Academy of Sciences, and School of Mathematical Sciences, University of Chinese Academy of Sciences, Beijing, 100190, China} \email{sun@math.ac.cn}



\begin{abstract}
We prove the  vanishing of certain low degree cohomologies of some induced representations. As an application, we determine certain low degree  cohomologies of congruence  groups.
\end{abstract}

 \maketitle


\section{Introduction}\label{s1}

\subsection{Vanishing of low degree cohomologies of some induced representations}
Let $G$ be a real reductive group, namely, it is a Lie group with
the following properties:
\begin{enumerate}
       \item[$\bullet$]
        its Lie algebra $\g$ is reductive;
                   \item[$\bullet$]
  $G$ has only finitely many connected components;
         \item[$\bullet$]
            there is a connected closed subgroup of $G$ with
  finite center whose Lie algebra equals  $[\g,\g]$.
\end{enumerate}
Here and henceforth, unless otherwise specified, we use the corresponding lowercase Gothic letter to indicate the 
Lie algebra of a Lie group. We use a subscript $\C$ to indicate the complexification of a real vector space. For example, $\g_\C$ is the complexified Lie algebra of $G$.


Let $F$ be an  irreducible finite-dimensional representation of $G$.  Define $r_{G,F}$ to be the smallest integer, if it exists,  such that  the continuous group cohomology 
\[
 \oH^{r_{G,F}}_{\mathrm{ct}}(G; F\otimes \pi)\neq \{0\}
\]
for some  infinite-dimensional irreducible unitary representations $\pi$ of $G$. If such an integer does not exist, we set $r_{G,F}=\infty$.  
See \cite[Section 2]{HM} for the definition of continuous cohomologies.

Define
\[
  r_G:=\min\{r_{G,F}\, : \, F \textrm{ is an irreducible  finite-dimensional  representation of $G$}\}.
\]
The quantity $r_G$ is calculated in \cite[Table 1]{En},  \cite[Theorem 2, Theorem 3]{Ku}, \cite[Table 8.2]{VZ} and \cite[Section 2.B]{LS2}. See also Tables \ref{tableg}, \ref{tableg2} and \ref{tableg3} of Section \ref{sectable}. 

Let $K$ be a maximal compact subgroup of $G$. 

\begin{examplesl}
(a). If  the Lie algebra $\g$ has no non-compact simple ideal, then  $G$ has no  infinite-dimensional  irreducible unitary representation. Thus in this case
$r_{G,F}=\infty$ for all $F$, and  $r_G=\infty$.

(b). If all highest weights of $F$ are regular, then $r_{G, F}\geq q_0(G)$, where
\be\label{rgfq}
q_0(G):= \frac{\dim G/K-(\rank \g_\C-\rank \k_\C)}{2}, 
\ee
and ``$\rank$" stands for the rank of a reductive complex Lie algebra. See \cite[
Proposition 4.2 and Proposition 4.4]{LS}. 

(c) If $G=\GL_n(\R)$ ($n\geq 2$), then $r_G=n-1$.

\end{examplesl}

Let $P$ be a  parabolic subgroup of $G$, namely the normalizer in $G$ of a  parabolic subalgebra of  $\g$. Denote by $N$ its 
 unipotent radical and write $L:=P/N$ for its Levi quotient.
 
 
 \begin{dfnl}\label{domch}
A positive character $\nu: L\rightarrow \R^\times_+$ is said to be dominant if 
\[
  \la \nu|_{\h_\C}, \alpha^\vee\ra\geq 0\quad \textrm{for all }\alpha\in \Delta(\h_\C, \n_{\C}).
\]
Here the compleixfied differential of $\nu: L\rightarrow \R^\times_+$ is denoted by  $\nu: \l_\C\rightarrow \C$,  $\h$ is a maximally split Cartan subalgebra of $\l$, and a splitting $L\rightarrow P$ is fixed so that $\h$ is viewed as a Cartan subalgebra of $\g$.  
 \end{dfnl}

 Here and henceforth, $\Delta(\h_\C, \n_{\C})$  denotes the subset of the root system $\Delta(\h_\C, \g_{\C})$ consisting of the roots attached to $\n_\C$, and $\alpha^\vee$ denotes the coroot corresponding to the root $\alpha$.  Note that $ \la \nu|_{\h_\C}, \alpha^\vee\ra$ is  a real number as $\nu$ is a positive character of $L$.

\begin{remarkl}
A splitting $L\rightarrow P$ exists and is unique up to the conjugation by a unique element of $N$. This implies that Definition \ref{domch} is independent of the choices of $\h$ and  the splitting $L\rightarrow P$. 
\end{remarkl}

The following theorem is the main  local result of this paper.

\begin{thml}\label{mainlocal}
Suppose  that $P$ is a proper parabolic subgroup of $G$. Let $\nu: L\rightarrow \R^\times_+$ be a dominant positive character. Then for every   irreducible unitarizable Casselman-Wallach representation $\sigma$ of $L$, and every  irreducible finite-dimensional   representation $F$ of $G$,
\[
   \oH_{\mathrm{ct}}^i(G;  F\otimes \Ind_P^G (\nu\otimes\sigma))=\{0\}\qquad \textrm{for all }i<r_G.
\]
\end{thml}

In this article, ``$\Ind$" always stands for the normalized smooth induction.  Recall that a representation of a real reductive group is called a Casselman-Wallach representation if it smooth, Fr\'echet, of moderate growth, and its Harish-Chandra module has finite length.  The reader is referred to \cite{Ca}, \cite[Chapter 11]{Wa2} or \cite{BK} for
details about Casselman-Wallach representations.

\begin{remarkl}\label{remark15}
In Theorem \ref{mainlocal}, if  we further assume  that all the highest weights of $F$ are regular, and $\sigma$ is tempered, then it is proved in \cite[Theorem 4.5]{LS} that 
\[
   \oH_{\mathrm{ct}}^i(G;  F\otimes \Ind_P^G (\nu\otimes\sigma))=\{0\}
   \]
for all $i<q_0(G)$ (see \eqref{rgfq}).


\end{remarkl}

Similar to $r_{G, F}$, we define $q_{G,F}$ to be the smallest integer, if it exists, such that  the continuous group cohomology 
\[
 \oH_{\mathrm{ct}}^{q_{G,F}}(G;  F\otimes \Ind_P^G (\nu\otimes\sigma))\neq \{0\}
\]
for some proper parabolic subgroups $P$ of $G$,  some unitarizable irreducible Casselman-Wallach representations $\sigma$ of $L$, and some dominant positive characters $\nu$ of $L$.
If no such integer exist, we set $q_{G,F}=\infty$. Define 
\[
  q_G:=\min\{q_{G,F}\, : \, F \textrm{ is a  finite-dimensional irreducible representations of $G$}\}.
\]
Theorem \ref{mainlocal} amounts to saying that 
\[
  r_G\leq q_G. 
\]

\subsection{Low degree cohomologies of congruenc groups}

Now we further assume that  $G=\mathsf G(\R)$, where $\mathsf G$  is a  connected reductive linear algebraic group defined over $\BQ$.  Let  $\Gamma$ be a congruence subgroup of $\mathsf G(\BQ)$, namely, $\Gamma$ has a finite index subgroup of the form  $\mathsf G(\BQ)\cap K_f$, where $K_f$ is an open compact subgroup of $\mathsf G(\A_f)$. Here $\A_f$ denotes the ring of finite adeles of $\BQ$, and $\mathsf G(\BQ)$ is viewed as a subgroup of $G$, as well as a subgroup of $\mathsf G(\A_f)$.    


Recall that $F$ is an irreducible finite dimensional representation of $G$. The group cohomology (also named the Eilenberg-MacLane cohomology)  $\oH^i(\Gamma;F)$ is of great interest in both geometry and arithmetic.  
If we viewed $\Gamma$ as a discrete group, then $\oH^i(\Gamma;F)$ equals the continuous cohomology $\oH_{\mathrm{ct}}^i(\Gamma;F)$.

Let $\mathsf A_\mathsf G$ denote the largest central split torus in $\mathsf G$, and write  $A_\mathsf G$ for the identity connected component of $\mathsf A_\mathsf G(\R)$. Then we have a decomposition
\be\label{decomg}
   G=A_\mathsf G \times  {^\circ \!G},
\ee
where 
\be\label{gr1}
 {^\circ \!G}:=\bigcap_{\chi\textrm{ is an algebraic character of $\mathsf G$ defined over $\BQ$}} \ker(\abs{\chi}: \mathsf G(\R)\rightarrow \R^\times_+).
\ee

Note that $\Gamma\subset  {^\circ \!G}$. By restriction of continuous cohomology, the embedding $F^\Gamma\hookrightarrow F$ induces a linear map
\be\label{maph}
  F^{\Gamma}\otimes \oH_{\mathrm{ct}}^i( {^\circ \!G} ; \C)= \oH_{\mathrm{ct}}^i( {^\circ \!G} ; F^{\Gamma})\rightarrow  \oH^i(\Gamma;F), \qquad(i\in \Z).
  \ee
Here and henceforth, a superscript group indicates the invariant vectors under the group action, and here we let both $\C$ and $F^{\Gamma}$ carry the trivial representation of ${^\circ \!G}$.

\begin{remarkl}
Note that $K\subset {^\circ \!G}$, and  
$$\oH_{\mathrm{ct}}^i( {^\circ\! G}; \C)=\Hom_{K}(\wedge^i ({^\circ \!\g_\C}/\k_\C), \C), $$
where the last $\C$ stands for the trivial representation of $K$. See \cite[Chapter II, Corollary 3.2]{BW} or \cite[Proposition 9.4.3]{Wa1}.

\end{remarkl}

At least when $F$ is the trivial representation, there are quite a few works which study the surjectivity and the injectivity of the map \eqref{maph}. See \cite{Ga, GH, Bo, Ya}, for examples. 
The following theorem heavily relies on the work of  Franke \cite[Sections 7.4 and 7.5]{Fr}. 

\begin{thml}\label{main0}
If $i\leq q_{G, F}$, then  the map \eqref{maph} is injective. If $i< \min \{r_{G,F}, q_{G, F}\}$, then  the map \eqref{maph} is a linear isomorphism. 

\end{thml}


Combining Theorem \ref{mainlocal} and Theorem \ref{main0}, we get the following result.  

\begin{thml}\label{main}
If  $i< r_G$, then the map \eqref{maph} is a linear isomorphism.
\end{thml}

\begin{remarksl}
(a)  When all the highest weights of $F$ are regular, it is proved by Li-Schwermer in \cite[Corollary 5.6]{LS} that  $\oH^i(\Gamma;F)=\{0\}$ when 
$
i<q_0(G)$ (see \eqref{rgfq}).

(b) Assume  that $F$ is the trivial representation. It was proved by Borel in \cite[Theorem 7.5]{Bo} that the map \eqref{maph} is injective when $i\leq c(\mathsf G)$, where  $c(\mathsf G)$ is a certain constant attached to $\mathsf G$ (see \cite[Section 9]{Bo}).  It was later proved by Yang in \cite{Ya} that the  map \eqref{maph} is injective when $i\leq 2c(\mathsf G)+1$. Based on the works of Garland-Hsiang \cite{GH} and Garland \cite{Ga}, Borel also proved in \cite[Theorem 7.5]{Bo} that the map \eqref{maph} is a linear isomorphism when $i\leq \min\{c(\mathsf G), m(G)\}$, where $m(G)$ is a certain constant attached to the real group $G$. In the works of Borel and Yang, $\Gamma$ is allowed to be any arithmetic group.

\end{remarksl}

\begin{examplel}
Suppose $\mathsf G=\GL(n)/_{\BQ}$ ($n\geq 2$). Then $c(\mathsf G)=\lfloor \frac{ n}{2}\rfloor-1$ (see  \cite[Section 9.2]{Bo}) and $m(G)=\frac{n+2}{4}$ (see \cite[Theorem 4.1]{KN}). 
Suppose $n\geq 6$ so that $\frac{n+2}{4}\leq \lfloor \frac{ n}{2}\rfloor-1$. Borel's result implies that 
\[
 \oH^i(\Gamma;\C)= \oH_{\mathrm{ct}}^i( {^\circ \!G} ; \C)\quad
 \textrm{
 for all $i\leq\lfloor\frac{n+2}{4}\rfloor$}. 
\]
It was pointed out by Franke \cite[Section 7.5]{Fr} that it would be interesting to get an improvement of  Borel's bound for certain groups. Theorem  \ref{main} implies that 
 \[  
      \oH^i(\Gamma;F)=
    F^{\Gamma}\otimes \oH_{\mathrm{ct}}^i( {^\circ \!G} ; \C)\quad \textrm{
 for all $i\leq n-2$},
    \]
    and 
    for all irreducible finite dimensional representations $F$ of $G$.  
Thus Theorem  \ref{main}  does improve Borel's bound.
\end{examplel}

\section{Preliminaries on representations and their cohomologies}

\subsection{Continuous cohomologies}

Let $G$ be a locally compact Hausdorff topological group. By a representation  of $G$, we mean a quasi-complete  Hausdorff locally convex topological vector space $V$ over $\C$, together with a  continuous linear action of $G$ on it. 

Using the $G$-intertwining continuous  linear maps as morphisms, all representations of $G$ clearly form a category. Using the strong injective resolutions (or called continuously injective resolutions) in this category, one defines the continuous  cohomology group $\oH^i_{\mathrm{ct}}(G; V)$ for every representation $V$ of $G$. This is a locally convex topological vector space  over $\C$ which may or may not be Hausdorff. See \cite[Section 2]{HM} or \cite{Bl} for details.

\subsection{Smooth cohomologies}
Now we suppose that $G$ is a Lie group. We say that a representation $V$ of $G$ is smooth if 
 for every $X\in \g$, the map
  \be\label{actgv}
  V\rightarrow V, \quad v\mapsto X.v:=\frac{\od}{\od \! t}|_{t=0}\, \exp(tX).v
\ee
   is well-defined and continuous. When this is the case, by using \eqref{actgv}, $V$ is naturally a representation of $\g_\C$.

Using the $G$-intertwining continuous  linear maps as morphisms, all smooth representations of $G$ also form a category. Using the strong injective resolutions (or called differentiably injective resolutions) in this category, we define the smooth  cohomology group $\oH^i_{\mathrm{sm}}(G; V)$ for every smooth representation $V$ of $G$. This is also a locally convex topological vector space  over $\C$ which may or may not be Hausdorff. See \cite[Section 5]{HM} for details.

The first assertion of the following theorem  is proved in  \cite[Theorem 5.1]{HM}, and the second one follows from \cite[Theorem 6.1]{HM}.  
\begin{thml}\label{fr}
Let $V$ be a smooth representation of $G$. Then there is an identification
\[
  \oH^{i}_{\mathrm{ct}}(G; V)=\oH^{i}_{\mathrm{sm}}(G; V),\qquad (i\in \Z)
\]
of topological vector spaces. 
If $G$ has only finitely many connected components, then 
\[
  \oH^{i}_{\mathrm{sm}}(G; V)=\oH^{i}(\g_\C, K; V), \qquad (i\in \Z)
\]
as topological vector spaces, where $K$ is a maximal compact subgroup of $G$. 

\end{thml}

Here $\oH^{i}(\g_\C, K; V)$ denotes the relative Lie algebra cohomology. See  \cite[(2.127)]{KV} for the explicit complex which  computes the relative Lie algebra cohomology.

\subsection{Quasi Casselman-Wallach representations}

Now suppose that $G$ is a real reductive group. Let $\oU(\g_\C)$ denote the universal enveloping algebra of $\g_\C$, whose center is denoted by $\oZ(\g_\C)$. \

\begin{dfnl}
A smooth representation $V$ of $G$ is called a quasi Casselman-Wallach representation if 
\begin{itemize}
  \item for every ideal $\mathfrak{J}$ of $\oU(\g_\C)^G$ of  finite co-dimension, the space 
     \be\label{vj0}
        V^{\mathfrak J}:=\{v\in V\, : \, X.v=0, \ \textrm{for all } X\in \mathfrak J\},
\ee
 which is $G$-stalbe,    is a Casselman-Wallach representation of $G$ under the  subspace topology;
  \item in the category of complex locally convex topological vector spaces, the natural map
  \[
    \varinjlim_{\mathfrak J} V^{ \mathfrak J}\rightarrow V
  \]
  is a topological linear isomorphism.
\end{itemize}

\end{dfnl}

We will be concerned with the continuous cohomologies of some quasi Casselman-Wallach representations.

\section{Preliminaries on automorphic forms}

\subsection{Spaces of automorphic forms}\label{sautof}

We now return to the notation of the Introduction. In particular, $\mathsf G$ is a connected reductive linear algebraic group over $\BQ$, and $G=\mathsf G(\R)$.

Recall that a subgroup $\Gamma$ of $\mathsf G(\BQ)$ is called an arithmetic subgroup if it is commensurable to the subgroup  $\eta^{-1}(\GL_k(\Z))\subset \mathsf G(\BQ)$, where $k\geq 0$ and $\eta: \mathsf G\rightarrow (\GL_k)/_{\BQ}$ is an embedding of algebraic groups over $\BQ$. This definition is independent of $k$ and $\eta$.  All congruence subgroups of $\mathsf G(\BQ)$ are arithmetic subgroups.

Let  $\mathsf P$ be a parabolic subgroup of $\mathsf G$. Denote by $\mathsf N_\mathsf P$ its unipotent radical, and write $\mathsf L_\mathsf P:=\mathsf P/\mathsf N_\mathsf P$ for its Levi quotient. Then $\mathsf L_\mathsf P$ is also a connected reductive linear algebraic group over $\BQ$. Write 
\[
 P=\mathsf P(\R), \quad N_{\mathsf P}:=\mathsf N_{\mathsf P}(\R)\quad \textrm{and}\quad L_{\mathsf P}:=\mathsf L_{\mathsf P}(\R).
\]

 Let $\Gamma_\mathsf P$ be an arithmetic subgroup of $\mathsf L_\mathsf P(\BQ)$. By abuse of notation, we use $\Gamma_\mathsf P N_{\mathsf P}$ to denote the pre-image of $\Gamma_{\mathsf P}$ under the quotient map
\[
  P\rightarrow  L_{\mathsf P}. 
\]

Let $\CA(\Gamma_\mathsf P N_{\mathsf P}\backslash G)$ denote the space of smooth automorphic forms on $\Gamma_\mathsf P N_{\mathsf P}\backslash G$. Recall that 
a complex valued smooth function $\phi$ on $\Gamma_\mathsf P N_{\mathsf P}\backslash G$  is called a smooth automorphic form if 
\begin{itemize}
\item 
$\phi$ is $\oZ(\g_\C)$-finite, namely the space $\{X.\phi\, : \, X\in \oZ(\g_\C)\}$ is finite dimensional (the algebra  $\oU(\g_\C)$ acts on the space of smooth functions by the differential of the right translations);
\item  
$\phi$ is uniformly of moderate growth, namely,  there exists a positive valued real algebraic function $\abs{\,\cdot\,}_\phi$ on $G$ with the following property:  for every $X\in \oU(\g_\C)$, there exists $C_{X}>0$ such that
\be
   \abs{(X. \phi)(g)}\leq C_{X} \cdot \abs{g}_\phi,\qquad \textrm{for all } g\in G.
\ee
\end{itemize}

 For every ideal $\mathfrak{J}$ of $\oZ(\g_\C)$ of  finite codimension, we have a space  $(\CA(\Gamma_\mathsf P N_{\mathsf P}\backslash G))^{\mathfrak J}$,  as defined in \eqref{vj0}. It carries the  action of $G$ by right translations. It also carries a natural Fr\'echet topology and is naturally a Casselman-Wallach representation of $G$, as explained in what follows. 
 Harish-Chandra's finiteness theorem (see \cite[Chapter I, \S 2, Theorem 1]{HC}) implies  that the space  of $K$-finite vectors in $(\CA(\Gamma_\mathsf P N_{\mathsf P}\backslash G))^{\mathfrak J}$ forms a finitely generated admissible $(\g_\C, K)$-module. Here and as before, $K$ is a maximal compact subgroup of $G$. Using this theorem and Casselman-Wallach's smooth globalization theorem (see \cite{Ca} or \cite[Corollary 11.6.8]{Wa2}), one shows that there is a positive valued real algebraic function $\abs{\,\cdot\,}_{\mathfrak J}$ on $G$ such that 
 \be\label{sn}
  \abs{\phi}_{X}:=\sup_{g\in G} \frac{\abs{(X. \phi)(g)}}{ \abs{g}_{\mathfrak J}}< \infty,
\ee
for all $\phi\in (\CA(\Gamma_\mathsf P N_{\mathsf P}\backslash G))^{\mathfrak J}$ and $X\in \oU(\g_\C)$. Under the seminorms $\{\abs{\,\cdot\,}_X\}_{X\in \oU(\g_\C)}$, $(\CA(\Gamma_\mathsf P N_{\mathsf P}\backslash G))^{\mathfrak J}$  becomes a Fr\'echet space, and is a Casselman-Wallach representation of $G$.  Moreover, the topology on $(\CA(\Gamma_\mathsf P N_{\mathsf P}\backslash G))^{\mathfrak J}$ is independent of 
the function  $\abs{\,\cdot\,}_{\mathfrak J}$ on $G$. 

We equip the space $\CA(\Gamma_\mathsf P N_{\mathsf P}\backslash G)$ with the inductive topology in the category of complex locally convex topological spaces:
\[
  \CA(\Gamma_\mathsf P N_{\mathsf P}\backslash G)=\varinjlim_{\mathfrak J} (\CA(\Gamma_\mathsf P N_{\mathsf P}\backslash G))^{\mathfrak J}.
\]
Then $\CA(\Gamma_\mathsf P N_{\mathsf P}\backslash G)$ is a quasi Casselman-Walach representation of $G$.

\subsection{The adelic setting}

Let $\A=\R\times \A_f$ denote the ring of adeles of $\BQ$. Then $\mathsf G(\A)=G\times \mathsf G(\A_f)$.
A representation $V$ of $\mathsf G(\A)$ is said to be smooth if 
\begin{itemize}
  \item it is smooth as a representation of $G$; and 
 \item  every vector in $V$ is fixed by an open compact subgroup of $\mathsf G(\A_f)$.
 \end{itemize}
 We define a quasi Casselman-Wallach representation of $\mathsf G(\A)$ to be a smooth representation $V$ of  $\mathsf G(\A)$ with the following properties:
\begin{itemize}
  \item for every open compact subgroup $K_f$ of $\mathsf G(\A_f)$, and every ideal $\mathfrak{J}$ of $\oZ(\g_\C)$ of  finite co-dimension, the space (with the subspace topology)
      \be\label{kfj}
        V^{K_f, \mathfrak J}:=\{v\in V\, : \, k.v=v, \ X.v=0, \ \textrm{for all }k\in K_f, \, X\in \mathfrak J\}
      \ee
     is a Casselman-Wallach representatiomation of $G$;
  \item in the category of complex locally convex topological vector spaces, the natural map
  \[
    \varinjlim_{K_f, \mathfrak J} V^{K_f, \mathfrak J}\rightarrow V
  \]
  is a topological linear isomorphism.
\end{itemize}

By abuse of notation, we use $\mathsf L_\mathsf P(\mathbb Q) \mathsf N_{\mathsf P}(\A)$ to denote the pre-image of $\mathsf L_\mathsf P(\mathbb Q)$  under the quotient map
\[
  \mathsf P(\A)\rightarrow \mathsf L_{\mathsf P}(\A). 
\]
Let $\mathcal A_\mathsf P(\mathsf G)$
denote the space of smooth automorphic forms on $(\mathsf L_\mathsf P(\mathbb Q) \mathsf N_{\mathsf P}(\A))\backslash \mathsf G(\A)$. Recall that a smooth function $\phi$ on  $(\mathsf L_\mathsf P(\mathbb Q) \mathsf N_{\mathsf P}(\A))\backslash \mathsf G(\A)$ is called a smooth automorphic form if 
\begin{itemize}
\item 
$\phi$ is right invariant under some open compact subgroup $K_f$ of $\mathsf G( \A_f)$;
\item 
for every $g_f\in \mathsf G( \A_f)$, the pullback of $\phi$ through the map 
\[
\begin{array}{rcl}
  \Gamma_{\mathsf P, g_f K_f} N_{\mathsf P}\backslash G&\rightarrow& (\mathsf L_\mathsf P(\mathbb Q) \mathsf N_{\mathsf P}(\A))\backslash \mathsf G(\A)/K_f,\\
   \textrm{the coset of $g\in G$} &\mapsto&\textrm{the double coset of $(g, g_f)\in \mathsf G(\A)$}\\
   \end{array}
\]
is a smooth automorphic form on $(\Gamma_{\mathsf P, g_f K_f} N_{\mathsf P})\backslash G$, where $\Gamma_{\mathsf P, g_f K_f} $ denotes the intersection of $\mathsf L_\mathsf P(\BQ)$ with the image $g_f K_f g_f^{-1} \cap \mathsf P(\A_f)$ under  the quotient map  
\[
  \mathsf P(\A_f)\rightarrow \mathsf L_\mathsf P(\A_f). 
\]
\end{itemize}

 By the discussion of Section \ref{sautof}, under the right translations, the space $(\mathcal A_\mathsf P(\mathsf G))^{K_f, \mathfrak J}$ (see \eqref{kfj}) is a Casselman-Wallach representation of $G$. Then using the inductive topology in the  category of locally convex topological vector space,
  \[
   \mathcal A_\mathsf P(\mathsf G)= \varinjlim_{K_f, \mathfrak J} (\mathcal A_\mathsf P(\mathsf G))^{K_f, \mathfrak J}
 \]
 is a quasi Casselman-Wallach representation of $\mathsf G(\A)$, under the right translations.

 \subsection{A decomposition of  $\mathcal A_\mathsf P(\mathsf G)$} 
 Write $\mathsf A_\mathsf P$ for the largest central split torus in $\mathsf L_\mathsf P$. We have a decomposition
\be\label{decomlp}
   \mathsf L_\mathsf P(\A)=A_\mathsf P\times \mathsf L^1_\mathsf P(\A),
\ee
where
\[
  \mathsf L^1_\mathsf P(\A):=\bigcap_{\chi\textrm{ is an algebraic character of $\mathsf L_{\mathsf P}$ defined over $\BQ$}} \ker\abs{\chi},
\]
and $A_\mathsf P$ denotes the identity connected component of $\mathsf A_\mathsf P(\R)$. Write $\a_\mathsf P$ for the Lie algebra of $A_\mathsf P$.

 The representation  $\mathcal A_\mathsf P(\mathsf G)$ also carries a locally finite linear action of $A_{\mathsf P}$ which commutes with the $\mathsf G(\A)$-action:
 \be\label{aca}
   (a.\phi)(x):=a^{-\rho_\mathsf P} \phi(ax),\quad a\in A_\mathsf P, \, \phi\in \mathcal A_\mathsf P(\mathsf G),\, x\in \mathsf P(\mathbb Q)\mathsf N_\mathsf P(\A)\backslash \mathsf G(\A),
 \ee
 where $\rho_\mathsf P\in \breve \a_\mathsf P$ denotes the half sum of the weights (with the multiplicities) associated to $\mathsf N_\mathsf P$. Here and henceforth, ``$\,\breve{\ }\,$" indicates the dual space.   The action \eqref{aca} differentiates  to a locally finite linear action of the commutative Lie algebra $\a_{\mathsf P,\C}:=\a_\mathsf P\otimes_\R \C$, and thus we have a generalized eigenspace decomposition
 \be\label{decome}
    \mathcal A_\mathsf P(\mathsf G)=\bigoplus_{\lambda\in \breve \a_{\mathsf P,\C}} \CA_\mathsf P(\mathsf G)_\lambda.
 \ee
For each $\phi\in  \mathcal A_\mathsf P(\mathsf G)$ and $\lambda\in \breve \a_{\mathsf P,\C}$, write $\phi_\lambda$ for the $\CA_\mathsf P(\mathsf G)_\lambda$-component of $\phi$ with respect to the decomposition \eqref{decome}.
The decomposition \eqref{decome} is $\mathsf G(\A)$-invariant. When $\mathsf P=\mathsf G$, we write $\CA(\mathsf G)$ for $\CA_{\mathsf G}(\mathsf G)$, and the decomposition \eqref{decome} is specified to
\be\label{decome2}
    \mathcal A(\mathsf G)=\bigoplus_{\lambda\in \breve \a_{\mathsf G,\C}} \CA(\mathsf G)_\lambda.
 \ee
In general, the map 
\[
  \phi\mapsto(g\mapsto (h\mapsto \phi(hg)\cdot h^{-\rho_\mathsf P})),\qquad  (g \in \mathsf G(\A), \,  h \in {\mathsf L}_{\mathsf P}(\mathbb Q)\backslash  {\mathsf L}_{\mathsf P}(\A))
\]
establishes an  isomorphism
\be\label{inducea}
  \CA_\mathsf P(\mathsf G)_{\lambda}\cong \Ind_{\mathsf P(\A)}^{\mathsf G(\A)} \left(\CA(\mathsf L_\mathsf P)_\lambda \right), \qquad (\lambda\in \breve \a_{\mathsf P,\C})
\ee
of Quasi Casselman-Wallach representations of $\mathsf G(\A)$. Here $h^{-\rho_\mathsf P}:=a_h^{-\rho_\mathsf P}$, with  $a_h$ denotes the projection of $h$ to $A_\mathsf P$ with respect to the decomposition \eqref{decomlp}.

\subsection{A lemma of Langlands}
Fix a pair $(\mathsf P_0, \mathsf A_0)$ such that $\mathsf P_0$ is a minimal parabolic subgroup of $\mathsf G$, and $\mathsf A_0$ is a maximal split torus in $\mathsf P_0$. Let 
 $\mathscr{P}$ denote the set of standard parabolic subgroups of $\mathsf G$, namely, parabolic subgroups containing $\mathsf P_0$. Suppose that  $\mathsf P\in \mathscr P$. View $\mathsf L_\mathsf P$ as an algebraic subgroup of $\mathsf P$ containing $\mathsf A_0$ so that 
\[
  \mathsf P=\mathsf L_\mathsf P\ltimes \mathsf N_\mathsf P.
\]

For simplicity, write $A_0:=A_{\mathsf P_0}$,  $\a_0:=\a_{\mathsf P_0}$.  Similar abbreviations will be used without further explanation.

 Denote by $\Delta_0^+\subset \breve \a_0$ the positive root system attached to the pair $(\mathsf P_0, \mathsf A_0)$.  More generally, write $\Delta_\mathsf P^+\subset \Delta_0^+$ for the set of roots associated to $\mathsf N_\mathsf P$. Given $\mathsf P'\in \mathscr P$ with $\mathsf P'\subset \mathsf P$, we have a decomposition
\[
  \a_{\mathsf P'}=\a_{\mathsf{P}}\oplus \a_{\mathsf{P}'}^{\mathsf{P}},\quad \textrm{and dually,}\quad \breve \a_{\mathsf P'}=\breve \a_{\mathsf{P}}\oplus \breve \a_{\mathsf{P}'}^{\mathsf{P}}
\]
where
\[
   \a_{\mathsf{P}'}^{\mathsf{P}}:=\a_{\mathsf P'}\cap \Span\{\alpha^\vee\, : \, \alpha\in \Delta_{\mathsf P'}^+\setminus \Delta_{\mathsf P}^+\}.
\]
Here and as before, for each $\alpha\in \Delta_0^+$, $\alpha^\vee\in \a_0$ denotes the corresponding coroot. In particular, we have
\[
  \a_{0}=\a_{\mathsf{P}}\oplus \a_0^{\mathsf{P}},\quad \textrm{and dually,}\quad \breve \a_0=\breve \a_{\mathsf{P}}\oplus \breve \a_0^{\mathsf{P}}.
\]

Define
\[
     \overline{\,^+\!\breve \a_{\mathsf{P}'}^{\mathsf{P}}}:=\textrm{the convex cone in $\breve \a_{\mathsf{P}'}^{\mathsf{P}}$ generated by } \{0\}\cup \{\alpha|_{\a_{\mathsf{P}'}^{\mathsf{P}}}\, : \, \alpha\in \Delta_{\mathsf P'}^+\setminus \Delta_{\mathsf P}^+ \}.  
\]
Let $^+\!\breve \a_{\mathsf P'}^{\mathsf P}$ denote the interior of $\overline{^+\!\breve \a_{\mathsf P'}^{\mathsf P}}$ in $\breve \a_{\mathsf P'}^{\mathsf P}$. 
Define
\begin{eqnarray*}
 \overline{{{\breve \a}_{\mathsf P}^+}}&:=&\{v\in  \breve \a_\mathsf{P}\, : \, \la  v, \alpha^\vee \ra\geq 0\textrm{ for all }\alpha\in \Delta_{\mathsf P}^+\}
\end{eqnarray*}
and
\[
  {\breve \a}_{\mathsf P}^+:=\{v\in  \breve \a_\mathsf{P}\, : \, \la  v, \alpha^\vee \ra> 0\textrm{ for all }\alpha\in \Delta_{\mathsf P}^+\}.
\]
Then ${\breve \a}_{\mathsf P}^+$ equals  the interior of  $\overline{{\breve \a}_{\mathsf P}^+}$ in ${\breve \a}_{\mathsf P}$.

Note that 
\[
  \overline{\breve \a_0^+}=\bigsqcup_{\mathsf P\in \mathscr P}  \breve \a_{\mathsf{P}}^+ \quad\textrm{and} \quad \overline{\,^+\breve \a_0^\mathsf G}=\bigsqcup_{\mathsf P\in \mathscr P} \,^+\breve{\a}_{\mathsf P}^{\mathsf G}.
\]
More generally, recall the following lemma which is due to Langlands (see \cite[Lemma IV.6.11]{BW} or \cite[Section 5.A.1]{Wa1}).
\begin{lem}
There is a decomposition
\be\label{langlands}
  \breve \a_0=\bigsqcup_{\mathsf P\in \mathscr P} \left(\breve \a_{\mathsf{P}}^+ - \overline{\,^+\!\breve \a_{0}^{\mathsf{P}}}\right).
\ee

\end{lem}

The decomposition \eqref{langlands} yields a map
\be\label{lp}
   \breve \a_0\rightarrow  \overline{\breve \a_{0}^+},\quad \nu \mapsto \nu_+
\ee
specified by requiring that $\nu_+\in \breve \a_{\mathsf{P}}^+$ and $\nu\in \nu_+ - \overline{\,^+\!\breve \a_{0}^{\mathsf{P}}}$, for some $\mathsf P\in \mathscr P$.


\subsection{Almost square integrable automorphic forms}
Put
\[
  \CA^2_\mathsf P(\mathsf G):=\{\phi\in \CA_\mathsf P(\mathsf G)\, : \, \phi(\,\cdot\,g)|_{\mathsf L_\mathsf P(\BQ)\backslash \mathsf L_\mathsf P^1(\A)}\textrm{ is square integrable for all } g\in \mathsf G(\A)\}.
\]
We use ``$\mathrm{Re}$" to indicates the real part in various context. 
Recall that an element $\phi$ of $\CA_\mathsf P(\mathsf G)$ belongs to $\CA^2_\mathsf P(\mathsf G)$ if and only if (see \cite[Lemma I.4.11]{MW})
\[
  \phi_{\mathsf P'}\in \bigoplus_{\lambda\in \breve \a_{\mathsf P',\C}, \,\mathrm{Re}(\lambda)\in  \breve \a_\mathsf P\, -\,^+\!\breve \a_{\mathsf P'}^{\mathsf P}} \CA_{\mathsf P'}(\mathsf G)_\lambda,\quad \textrm{for all $\mathsf P'\in \mathscr P$ with $\mathsf P'\subset \mathsf P$,}
\]
where   $\phi_{\mathsf P'}$ denotes the constant term of $\phi$ along $\mathsf P'$, namely 
\[
  \phi_{\mathsf P'}(g):=\int_{\mathsf N_{\mathsf P'}(\BQ)\backslash \mathsf N_{\mathsf P'}(\A)} \phi(ng) \od\! n
\]
for all $g\in \mathsf G(\A)$. Here $\od\!n$ is the invariant measure with total volume $1$.

 As a variation of  $\CA^2_\mathsf P(\mathsf G)$, we define 
\begin{eqnarray*}
   && \CA^{\bar 2}_\mathsf P(\mathsf G) \\
   \!&:=\!&\! \!\left \{\phi\in \CA_\mathsf P(\mathsf G)\, : \, \phi_{\mathsf P'}\in \bigoplus_{\lambda\in \breve \a_{\mathsf P',\C}, \,\mathrm{Re}(\lambda)\in \breve \a_\mathsf P\,-\,\overline{^+\!\breve \a_{\mathsf P'}^{\mathsf P}}} \CA_{\mathsf P'}(\mathsf G)_\lambda,\quad \textrm{for all $\mathsf P'\in \mathscr P$ with $\mathsf P'\subset \mathsf P$}\right \},
\end{eqnarray*}
and call it the space of almost square integrable automorphic forms.

It is clear that both $\CA^2_\mathsf P(\mathsf G)$ and $\CA^{\bar 2}_\mathsf P(\mathsf G)$ are are closed subspaces of $\CA_\mathsf P(\mathsf G)$. Moreover, they are $\mathsf G(\A)\times A_\mathsf P$-subrepresentations. 

The decomposition \eqref{decome} induces generalized eigenspace decompositions
\[
  \CA^2_\mathsf P(\mathsf G)=\bigoplus_{\lambda\in \breve \a_{\mathsf P,\C}} \CA^2_\mathsf P(\mathsf G)_\lambda\quad\textrm{and}\quad
  \CA^{\bar 2}_\mathsf P(\mathsf G)=\bigoplus_{\lambda\in \breve \a_{\mathsf P,\C}} \CA^{\bar 2}_\mathsf P(\mathsf G)_\lambda,
\]
and the isomorphism \eqref{inducea} induces isomorphisms
\be\label{ind2}
 \CA^{2}_\mathsf P(\mathsf G)_{\lambda}\cong \Ind_{\mathsf P(\A)}^{\mathsf G(\A)} \left(\CA^{2}(\mathsf L_\mathsf P)_\lambda \right)\quad \textrm{and}\quad \CA^{\bar 2}_\mathsf P(\mathsf G)_{\lambda}\cong \Ind_{\mathsf P(\A)}^{\mathsf G(\A)} \left(\CA^{\bar 2}(\mathsf L_\mathsf P)_\lambda \right),
 \ee
for all  $\lambda\in \breve \a_{\mathsf P,\C}$.
 
In what follows we describe the representation $\CA^{\bar 2}(\mathsf G)$ in terms of square integrable automorphic forms of the Levi factors. See \cite[Section 6]{Fr} for more  details.  
 View $\mathscr P$ as a category so that a morphism from $\mathsf P'\in \mathscr P$ to $\mathsf P''\in \mathscr P$ is defined to be an element of the set
\[
  \mathsf L_{\mathsf P''}(\BQ)\backslash \{\tilde w\in \mathsf G(\BQ)\, : \, \tilde w \mathsf L_{\mathsf P'} \tilde w^{-1}=\mathsf L_{\mathsf P''}\}/\mathsf L_{\mathsf P'}(\BQ),
\]
and the composition in the category in the one given by the group structure of $\mathsf G(\BQ)$. Then $\mathscr P$ is in fact a groupoid.  

Define a functor from $\mathscr P$ to the category of smooth representations of $\mathsf G(\A)$ as follows. It sends $\mathsf P'\in \mathscr P$ to the representation
\[
\CA^2_{\mathsf P'}(\mathsf G)_{\mathrm{Im}}:=\bigoplus_{\lambda\in \breve \a_{\mathsf P',\C},\,\mathrm{Re}(\lambda|_{\a_{\mathsf P'}^{\mathsf G}})=0}\CA^2_{\mathsf P'}(\mathsf G)_\lambda,
\]
and it sends a morphism $w: \mathsf P'\rightarrow \mathsf P''$ to the intertwining operator
\be\label{mw}
  M_{w}: \CA^2_{\mathsf P'}(\mathsf G)_{\mathrm{Im}}\rightarrow \CA^2_{\mathsf P''}(\mathsf G)_{\mathrm{Im}}.
      \ee
This is indeed a functor by the functional equation of intertwining operators (see \cite[Section 6]{La}).
Recall that the intertwining operator \eqref{mw} is obtained by the meromorphic continuation of the following family of  intertwining operators of  convergent integrals:
\be\label{mw2}
  \begin{array}{rcl}
  M_{w,\lambda'}: \CA^2_{\mathsf P'}(\mathsf G)_{\lambda'}&\rightarrow &\CA^2_{\mathsf P''}(\mathsf G)_{w.\lambda'},\\
 \phi&\mapsto & \left(g\mapsto \int_{(\mathsf N_{\mathsf P''}\cap (\tilde w \mathsf N_{\mathsf P'}\tilde w^{-1}))(\A)\backslash \mathsf N_{\mathsf P''}(\A)} \phi(\tilde w^{-1} ng)\od\! n\right),
 \end{array}
\ee
where $\lambda'\in \breve \a_{\mathsf P',\C}$ with
\be\label{deflamda}
  \la \mathrm{Re}(\lambda'-\rho_{\mathsf P'})|_{\a_{\mathsf P'}^\mathsf G}, \alpha^\vee\ra>0\quad \textrm{ for all }\alpha\in \Delta_{\mathsf P'}^+.
\ee
In \eqref{mw2}, $\tilde w\in \mathsf G(\BQ)$ denotes a representative of $w$, and $\od \!n$ denotes the quotient of the normalized Haar measures (the Haar measure on $\mathsf N_{\mathsf P''}(\A)$ is normalized so that $\mathsf N_{\mathsf P''}(\BQ)\backslash \mathsf N_{\mathsf P''}(\A)$ has volume $1$, and similarly for other unipotent groups).

The theory of Eisenstein series gives a $\mathsf G(\A)$-intertwining continuous linear map
\be\label{eis}
  E_{\mathsf P'}: \CA^2_{\mathsf P'}(\mathsf G)_{\mathrm{Im}}\rightarrow \CA(\mathsf G).
\ee
Recall that the map \eqref{eis} is obtained by the meromorphic continuation of the following family of  convergent Eisenstein series:
\[
  \begin{array}{rcl}
  E_{\mathsf P',\lambda'}: \CA^2_{\mathsf P'}(\mathsf G)_{\lambda'}&\rightarrow &\CA(\mathsf G),\\
 \phi&\mapsto & \left(g \mapsto \sum_{\delta\in \mathsf P'(\BQ) \backslash \mathsf G(\BQ)} \phi(\delta g)\right),
 \end{array}
\]
where $\lambda'$ is as in \eqref{deflamda}. By the functional equation of Eisenstein series, \eqref{eis} for various $\mathsf P'$ yields a
$\mathsf G(\A)$-intertwining continuous linear map
\be\label{eis2}
 E: \varinjlim_{\mathsf P'\in \mathscr P} \CA^2_{\mathsf P'}(\mathsf G)_{\mathrm{Im}}\rightarrow \CA(\mathsf G).
\ee

The following theorem is a reformulation of a special case of \cite[Theorem 14]{Fr}.  
\begin{thml}\label{thmbar2}
The map \eqref{eis2} induces a $\mathsf G(\A)$-intertwining topological linear isomorphism
\be\label{eis3}
 \varinjlim_{\mathsf P'\in \mathscr P} \CA^2_{\mathsf P'}(\mathsf G)_{\mathrm{Im}}\cong \CA^{\bar 2}(\mathsf G).
\ee
\end{thml}

As a reformulation of \eqref{eis3}, we have
\be\label{eis4}
 \CA^{\bar 2}(\mathsf G)\cong \CA^{2}(\mathsf G)\oplus \left(\bigoplus_{\{\mathsf P\}, \mathsf P\neq \mathsf G} (\CA^2_{\mathsf P}(\mathsf G)_{\mathrm{Im}})^{\Aut_\mathscr P(\mathsf P)}\right),
\ee
where $\{\mathsf P\}\subset \mathscr P$ denotes the associated class of $\mathsf P$, namely the isomorphism class of $\mathsf P$ in the groupoind $\mathscr P$, and $\Aut_\mathscr P(\mathsf P)$ denotes the finite group of the automorphisms of the object $\mathsf P$ in the groupoid $\mathscr P$.

\subsection{Franke's filtration}

Write
\[
\rho_0^\vee:=\frac{1}{2} \sum_{\alpha\in \Delta_0^+, \, \alpha/2 \notin \Delta_0^+} \alpha^\vee \in \a_0
\]
so that  (see \cite[Proposition 2.69]{Kn} for example)
\[
  \la \rho_0^\vee, \alpha\ra=1\quad \textrm{for every simple root in $\Delta_0^+$.}
\]
  Following Franke \cite{Fr}, for each real number $t\geq 0$, define
\[
  \CA(\mathsf G)_{[\leq t]}:=\left \{\phi\in \CA(\mathsf G)\, : \, \phi_{\mathsf P}\in \bigoplus_{\lambda\in \breve \a_{\mathsf P,\C}, \,\, \la (\mathrm{Re}(\lambda))_+ , \rho_0^\vee\ra \leq t } \CA_\mathsf P(\mathsf G)_\lambda,\,\textrm{ for all $\mathsf P\in \mathscr P$} \right \},
\]
\[
  \CA(\mathsf G)_{[< t]}:=\left \{\phi\in \CA(\mathsf G)\, : \, \phi_{\mathsf P}\in \bigoplus_{\lambda\in \breve \a_{\mathsf P,\C}, \,\, \la (\mathrm{Re}(\lambda))_+ , \rho_0^\vee\ra <t } \CA_\mathsf P(\mathsf G)_\lambda,\,\textrm{ for all $\mathsf P\in \mathscr P$} \right \},
\]
and
\[
   \CA(\mathsf G)_{[t]}:=\frac{\CA(\mathsf G)_{[\leq t]}}{\CA(\mathsf G)_{[<t]}},
\]
where $\phi_\mathsf P\in \CA_\mathsf P(\mathsf G)$ denotes the constant term of $\phi$ along $\mathsf P$, as before. We remark that these spaces are independent of the pair $(\mathsf P_0, \mathsf A_0)$. It is clear that
\[
  \CA(\mathsf G)_{[\leq 0]}=\CA^{\bar 2}(\mathsf G)\quad\textrm{and}\quad \CA(\mathsf G)_{[< 0]}=\{0\}.
\]

Define a linear map
\be\label{fil0}
   \begin{array}{rcl}
      \CA(\mathsf G)_{[\leq t]}&\rightarrow & \bigoplus_{\mathsf P\in \mathscr P} \CA_\mathsf P(\mathsf G),\smallskip \\
        \phi&\mapsto & \sum_{\mathsf P\in \mathscr P}  \sum_{\lambda\in \breve \a_{\mathsf P,\C}, \,\Re(\lambda)\in \breve \a_\mathsf P^+,\,\la \Re(\lambda), \rho_0^\vee\ra=t}(\phi_\mathsf P)_\lambda.
   \end{array}
\ee

The following theorem follows from \cite[Theorem 14]{Fr} and its proof.  

\begin{thml}\label{frf}
The linear map \eqref{fil0} induces an isomorphism 
\be\label{isogr}
  \CA(\mathsf G)_{[t]}\cong  \bigoplus_{\mathsf P\in \mathscr P} \, \bigoplus_{\lambda\in \breve \a_{\mathsf P,\C}, \,\Re(\lambda)\in \breve \a_\mathsf P^+,\,\la \Re(\lambda), \rho_0^\vee\ra=t}\CA_\mathsf P^{\bar 2}(\mathsf G)_\lambda
\ee
of quasi Casselman-Wallach representations of $\mathsf G(\A)$. 
\end{thml}

\subsection{Franke's filtration with fixed generalized infinitesimal character}

Write
\[
   \CA(\mathsf G)=\bigoplus_{\chi: \oZ(\g_\C)\rightarrow \C} \CA(\mathsf G)_\chi
\]
for  the generalized eigenspace decomposition with respect to the action of $\oZ(\g_\C)$. 
For each character $\chi: \oZ(\g_\C)\rightarrow \C$, 
\be\label{filchi}
  \CA(\mathsf G)_{\chi,\,[\leq\, t]}:=\CA(\mathsf G)_{\chi}\cap \CA(\mathsf G)_{[\leq\, t]}, \quad (t\in \R, t\geq 0),
\ee
defines an increasing filtration of $\CA(\mathsf G)_{\chi}$. Let $\h_{0}$ denote a  maximally split Cartan subalgebra of the Lie algebra of  of $\mathsf L_0(\R)$. 
Then $\h_0\supset \a_0$ and hence $\h_{0, \C}\supset \a_{0, \C}$. Define
\[
  \Xi_\chi:=\{\mu|_{\a_{0,\C}}\,:\,\mu \in \breve \h_{0,\C}\textrm{ is a Harish-Chandra parameter of $\chi$}  \}.
\]
This is independent of the choice of $\h_{0}$.
Put
\[
  T_\chi:=\{\la (\Re(\lambda|_{\a_\mathsf P}))_+, \rho_0^\vee\ra\, : \, \mathsf P\in \mathscr P,\,\lambda\in \Xi_\chi\}.
\]
It is a finite set containing $0$.

\begin{lem}\label{infifil0}
Suppose that $\phi\in \CA(\mathsf G)_{\chi}$ and $\phi_{\mathsf P, \lambda}\neq 0$ for some $\mathsf P\in \mathscr P$ and $\lambda\in \breve \a_{\mathsf P,\C}$. Then 
\[
 \la (\Re(\lambda|_{\a_\mathsf P}))_+, \rho_0^\vee\ra\in T_\chi.
\]
\end{lem}
\begin{proof}
Recall that the normalized parabolic inductions for real reductive groups preserve the infinitesimal characters. This implies that
\[
 \lambda=\mu|_{\a_{\mathsf P,\C}} \ \  \textrm{for some Harish-Chandra parameter $\mu\in \breve \h_{0,\C}$ of $\chi$.}
\]
Hence the lemma follows. 
\end{proof}

Lemma \ref{infifil0} implies that the filtration \eqref{filchi} is finite. More precisely, Lemma \ref{infifil0} implies  the following result. 
\begin{lem}\label{infifil}
Let $t$ be a nonnegative real number. Then 
\[
  \CA(\mathsf G)_{\chi,\,[\leq\, t]}=\CA(\mathsf G)_{\chi,\,[\leq\, t']}, \quad \textrm{where }t':=\max\{t_1\in T_\chi\, : \, t_1\leq t\}.
\]
\end{lem}

\section{Low degree cohomologies}

We continue with the notation of the last section.

\subsection{Low degree cohomologies of the space of automorphic forms}
 Write $\mathsf T_\mathsf P$ for the maximal anisotropic  central  torus in $\mathsf L_\mathsf P\cap \mathsf G^{\mathrm{der}}$, where $\mathsf G^{\mathrm{der}}$ denotes the derived subgroup of $\mathsf G$. Write
 $\mathsf C_{\mathsf P}$ for the the maximal  split torus of the real algebraic torus $\mathsf T_\mathsf P\times_{\mathrm{spec}(\BQ)} \mathrm{spec}(\R)$.

We say that an irreducible Cassleman-Wallach repesentation $\sigma$ of $L_\mathsf P=\mathsf L_\mathsf P(\R)$ has a $\g$-real infinitesimal character if 
\[
  \la \lambda_\sigma, \alpha^\vee\ra \in \R\quad \textrm{for all root $\alpha\in \Delta(\h_\C, \g_\C)$}, 
\]
where $\h$ is a Cartan subalgebra of the Lie algebra of $ L_{\mathsf P}$, and $\lambda_\sigma\in \breve \h_\C$ is a Harish-Chandra parameter of the infinitesimal character of $\sigma$. This definition is independent of $\h$ and $\lambda_\sigma$.

\begin{lem}\label{vanisht}
Let $\sigma$ be an irreducible Cassleman-Wallach repesentation of $ L_\mathsf P$ with $\g$-real infinitesimal character. If it occurs as an $L_\mathsf P$-subquotient of  $\CA(\mathsf L_{\mathsf P})$, then the identity connected component $\mathsf C_{\mathsf P}(\R)_0$ of $\mathsf C_{\mathsf P}(\R)$  acts trivially on $\sigma$.
\end{lem}
\begin{proof}
The group $\mathsf C_{\mathsf P}(\R)_0$ acts on $\sigma$ by a character $\chi$. The automorphic condition implies that $\chi$ is unitary, and the condition of $\g$-real infinitesimal character implies that $\chi$ is real. Hence $\chi$ must be trivial. 
\end{proof}

As in the Introduction, let $F$ be an irreducible finite dimensional  representation of $G=\mathsf G(\R)$. Recall the quantity $q_{G,F}$ from the Introduction. 

\begin{lem}\label{vaind}
Assume that $\mathsf P\neq \mathsf G$. Let $\lambda\in \breve \a_{\mathsf P,\C}$ with $\Re(\lambda)\in \overline{\breve \a_{\mathsf{P}}^+}$. Let $\sigma$ be an irreducible Cassleman-Wallach representation of $L_\mathsf P$ which occurs as an irreducible $L_\mathsf P$-subquotient  of $\CA^{\bar 2}(\mathsf L_\mathsf P)_\lambda$. Then
\be \label{chohi}
   \oH_\mathrm{ct}^i(G;  F\otimes \Ind_{P}^{G}\sigma)=0\quad  \textrm{ for all $i<q_{G, F}$.}
\ee
\end{lem}
\begin{proof}
If the infinitesimal character of $\sigma$ is not $\g$-real, then the cohomology space of \eqref{chohi} vanish. So we assume that $\sigma$ has a $\g$-real infinitesimal character. Then in view of the definition of $q_{G,F}$, the lemma follows from Lemma \ref{vanisht}.
\end{proof}

\begin{lem}\label{val2}
Assume that $\mathsf P\neq \mathsf G$. Let $\lambda\in \breve \a_{\mathsf P,\C}$ with $\Re(\lambda)\in \overline{\breve \a_{\mathsf{P}}^+}$. Then
\[
   \oH_\mathrm{ct}^i(G;  F\otimes \CA^{\bar 2}_{\mathsf P}(\mathsf G)_\lambda)=0\quad  \textrm{ for all $i<q_{G,F}$.}
\]
\end{lem}
\begin{proof}
Note that $\CA^{\bar 2}_{\mathsf P}(\mathsf G)_\lambda$ has a filtration $V_0\subset V_1\subset V_2\subset \cdots$ such that every associated graded representation is a direct sum of representations of the form
\[
 \Ind_{\mathsf P(\A)}^{\mathsf G(\A)} \sigma,
\]
where  $\sigma$ is an irreducible $\mathsf L_\mathsf P(\A)$-subrepresentation of $\CA^{\bar 2}(\mathsf L_{\mathsf P})_\lambda$. Thus the lemma follows from Lemma \ref{vaind}.
\end{proof}

\begin{lem}\label{val2234}
For all $i\in \Z$, the natural linear map
\[
  \oH_\mathrm{ct}^i( G;  F\otimes \CA^{2}(\mathsf G))\rightarrow  \oH_\mathrm{ct}^i(G;  F\otimes \CA^{\bar 2}(\mathsf G))
\]
is injective, and it is an isomorphism if $i<q_{G, F}$.
\end{lem}
\begin{proof}
In view of \eqref{eis4}, this is implied by Lemma \ref{val2}.
\end{proof}

\begin{lem}\label{val2235}
The natural linear map
\[
  \oH_\mathrm{ct}^i( G;  F\otimes \CA^{\bar 2}(\mathsf G))\rightarrow  \oH_\mathrm{ct}^i(G;  F\otimes \CA(\mathsf G))
\]
is injective if $i\leq q_{G, F}$, and is bijective if $i<q_{G,F}$. 

\end{lem}
\begin{proof}

This follows by Theorem \ref{frf}, Lemma \ref{infifil}
and  Lemma \ref{val2}.
\end{proof}

Combining Lemmas \ref{val2234} and \ref{val2235}, we get the following result.
\begin{prpl}\label{val223t}
The natural linear map
\[
  \oH_\mathrm{ct}^i( G;  F\otimes \CA^{2}(\mathsf G))\rightarrow  \oH_\mathrm{ct}^i(G;  F\otimes \CA(\mathsf G))
\]
is injective if $i\leq q_{G, F}$, and is bijective if $i<q_{G,F}$. 

\end{prpl}

\subsection{Automorphic forms on $\Gamma\backslash {^\circ \!G} $}

Recall from \eqref{decomg} the decomposition 
\[
   G=A_\mathsf G \times  {^\circ \!G}.
\]
Let $\Gamma$ be an arithmetic   subgroup of $\mathsf G(\BQ)$. Then $\Gamma \subset {^\circ \!G}$ and
\[
  \Gamma\backslash G=A_\mathsf G\times ( \Gamma\backslash {^\circ \!G}).
\] 
As in Section \ref{sautof} , we say that a smooth function on $\Gamma\backslash {^\circ \!G}$ is a smooth automorphic form if it is $\oZ({^\circ \!\g_\C})$-finite, and is  uniformly of moderate growth. Denote by $\CA(\Gamma\backslash {^\circ \!G})$ the space of smooth automorphic forms on $\Gamma\backslash {^\circ \!G}$. As in Section \ref{sautof}, it is a quasi Casselman-Wallach representation of ${^\circ \!G}$. 

Write 
\[
\CA(A_\mathsf G):=\{\phi\in \con^\infty(A_\mathsf G)\, :\, \phi \textrm{ is  $A_\mathsf G$-finite}  \}, 
\] 
where $ \phi$  is  $A_\mathsf G$-finite means that $\{ \phi(\,\cdot\, a)\}_{a\in A_\mathsf G}$ spans a finite dimensional subspace of $\con^\infty(A_\mathsf G)$. This is a quasi-Cassleman-Wallach representation of $A_\mathsf G$ under the right translations, and under the finest locally convex topology. 

It is clear that 
\[
   \CA(\Gamma\backslash G)=\CA(A_\mathsf G)\otimes \CA(\Gamma\backslash {^\circ \!G}).
\]
Fix a ${^\circ \!G}$-invariant positive Borel measure on  $\Gamma\backslash {^\circ \!G}$. It has finite total volume. 
Write $\CA^2(\Gamma\backslash {^\circ \!G})\subset \CA(\Gamma\backslash {^\circ \!G})$ for the subspace of square integrable functions with respect to this invariant measure, and write
\[
  \CA^2(\Gamma\backslash G)=\CA(A_\mathsf G)\otimes \CA^2(\Gamma\backslash {^\circ \!G})\subset  \CA(\Gamma\backslash G).
\]

\begin{lem}\label{l2c}
By restriction of continuous cohomology, the restriction map
\[
 \CA(\Gamma\backslash G)\rightarrow \CA(\Gamma\backslash {^\circ \!G})
\]
induces a linear isomorphism
\[
  \oH_{\mathrm{ct}}^i(G; F\otimes \CA(\Gamma\backslash G))\cong \oH_{\mathrm{ct}}^i({^\circ \!G}; F\otimes \CA(\Gamma\backslash {^\circ \!G})),\quad (i\in \Z). 
\]
Likewise the restriction map
\[
 \CA^2(\Gamma\backslash G)\rightarrow \CA^2(\Gamma\backslash {^\circ \!G})
\]
induces a linear isomorphism
\[
  \oH_{\mathrm{ct}}^i(G; F\otimes \CA^2(\Gamma\backslash G))\cong \oH_{\mathrm{ct}}^i({^\circ \!G}; F\otimes \CA^2(\Gamma\backslash {^\circ \!G})),\quad (i\in \Z).
\]
\end{lem}

\begin{proof}
The group $A_\mathsf G$ acts on $F$ through a character which we denote by $\chi_F$. Using the K\"{u}nneth formula, the lemma then follow form the fact that
\[
  \oH_{\mathrm{ct}}^i(A_\mathsf G; \chi_F\otimes \CA(A_\mathsf G))=\left\{
                                                       \begin{array}{ll}
                                                         \C, & \hbox{if $i=0$;} \\
                                                         0, & \hbox{otherwise.}
                                                       \end{array}
                                                     \right.
\]
\end{proof}

\begin{prpl}\label{val223gg} 
Suppose that $\Gamma$ is a congruence subgroup of $\mathsf G(\BQ)$. 
Then the natural linear map
\[
  \oH_\mathrm{ct}^i(  {^\circ \!G};  F\otimes \CA^2(\Gamma\backslash {^\circ \!G}))\rightarrow  \oH_\mathrm{ct}^i( {^\circ \!G};  F\otimes \CA(\Gamma\backslash {^\circ \!G}))
\]
is injective if $i\leq q_{G, F}$, and is bijective if $i<q_{G,F}$. 

\end{prpl}
\begin{proof}
Without loss of generality, assume that $\Gamma=\mathsf G(\BQ)\cap K_f$, where $K_f$ is an open compact subgroup of $\mathsf G(\A_f)$. Then
\[
 ( \CA(\mathsf G))^{K_f}=\bigoplus_{j=1}^k\CA(\Gamma_j\backslash G),
\]
where $k$ is the cardinality of the finite set $\mathsf G(\BQ)\backslash \mathsf G(\A_f)/K_f$, and $\Gamma_j$'s are certain congruence subgroups of $\mathsf G(\BQ)$ with $\Gamma_1=\Gamma$.  Moreover, 
\[
  (\CA^2(\mathsf G))^{K_f}=\bigoplus_{j=1}^k\CA^2(\Gamma_j\backslash G).
\]
Therefore, Proposition \ref{val223t} implies that the natural linear map
\[
  \oH_\mathrm{ct}^i( G;  F\otimes \CA^{2}(\Gamma\backslash G))\rightarrow  \oH_\mathrm{ct}^i(G;  F\otimes \CA(\Gamma\backslash G))
\]
is injective if $i\leq q_{G, F}$, and is bijective if $i<q_{G,F}$. The proposition then follows by using Lemma \ref{l2c}.

\end{proof}

\subsection{Cohomologies of congruence groups} 

Let $\Gamma$ be an arithmetic   subgroup of $\mathsf G(\BQ)$ as before. By Shapiro's Lemma, 
\be\label{shapiro}
  \oH^i(\Gamma;F)=\oH_{\mathrm{ct}}^i({^\circ \!G}, \Ind_\Gamma^{^\circ \!G} F)=\oH_{\mathrm{ct}}^i({^\circ \!G}, F\otimes \con^\infty(\Gamma\backslash {^\circ \!G})),\quad (i\in \Z) 
\ee
where the second  identification is induced by the isomorphism
\[
\begin{array}{rcl}
  F\otimes \con^\infty(\Gamma\backslash {^\circ \!G})&\rightarrow &\Ind_\Gamma^{^\circ \!G} F, \\
  v\otimes \phi&\mapsto & (g\mapsto \phi(g) g.v).
  \end{array}
\]

\begin{remarkl}
 The cohomology group $\oH^i(\Gamma;F)$  has a topological interpretation, as explained in what follows. 
 Set
\[
  X_\Gamma:= \Gamma\backslash G/K. 
\]
For every open subset $U$ of $X_\Gamma$, write
\[
  F_{[\Gamma]}(U):=\{\textrm{$\Gamma$-equivariant locally constant function $f: \tilde U\rightarrow F$}\},
\]
where $\tilde U$ denotes the pre-image of $U$ under the quotient map $G/K\rightarrow X_\Gamma$. Then $F_{[\Gamma]}$ is obviously a sheaf of complex vector spaces on $X_\Gamma$. 
Moreover, its sheaf cohomology agrees with the group cohomology, namely there is a vector space identification
\be\label{hgammaf1}
  \oH^i(\Gamma;F)=\oH^i(X_\Gamma;F_{[\Gamma]}).
\ee
See \cite[Chapter VII]{BW} for more details. 
\end{remarkl}

Recall the following important result of Franke \cite[Theorem 18]{Fr} (see also \cite[Section 13]{Sc}). 

\begin{thml}\label{frgk}
Suppose that $\Gamma$ is a congruence subgroup of $\mathsf G(\BQ)$. Then the embedding $\CA(\Gamma\backslash {^\circ \!G})\hookrightarrow \con^{\,\infty}(\Gamma\backslash {^\circ \!G})$ induces an identification
\[
  \oH_{\mathrm{ct}}^i({^\circ \!G}; F\otimes \CA(\Gamma\backslash {^\circ \!G}))=\oH_{\mathrm{ct}}^i({^\circ \!G}; F\otimes \con^{\,\infty}(\Gamma\backslash {^\circ \!G}))
\]
of vector spaces. 
\end{thml}

By combining \eqref{shapiro} and Theorem \ref{frgk}, we get an identification 
\be\label{frgk22}
   \oH^i(\Gamma;F)= \oH_{\mathrm{ct}}^i({^\circ \!G}; F\otimes \CA(\Gamma\backslash {^\circ \!G})). 
\ee

\subsection{Finite dimensional representations in $ \CA^2(\Gamma\backslash {^\circ \!G})$}

In this subsection and the next one, we suppose that $\Gamma$ is a congruence subgroup of $\mathsf G(\BQ)$. 

\begin{lem}\label{fixgamma0}
There exists a cocompact closed subgroup $\Gamma'$ of ${^\circ \!G}$ such that for every irreducible finite dimensional representation of $F'$ of $G$, $\Gamma'$ acts trivially on $F'^{\Gamma}$.
\end{lem}
\begin{proof}
Without loss of generality, we assume that  $\mathsf G$ is either  an algebraic torus over $\BQ$, or a simply connected, connected, $\BQ$-simple linear algebraic group over $\BQ$. If $\mathsf G$ is an algebraic torus or an anisotropic simple algebraic group, then $\Gamma$ is cocompact in ${^\circ \!G}$ and there is nothing to prove. So we further assume that   $\mathsf G$ is an isotropic simple algebraic group. Then  $\Gamma$ is Zariski dense in $\mathsf G(\C)$, and the representation of $G$ on $F'$ extends to an algebraic representation of $\mathsf G(\C)$.  Thus the group $G$ fixes $F'^{\Gamma}$.
\end{proof}

As before, $\breve F$ denotes the dual representation of $F$.

\begin{lem}\label{auto2}
The image of every homomorphism in  $\Hom_{^\circ \!G}(\breve F, \con^\infty(\Gamma\backslash {^\circ \!G}))$  is contained in $\CA^2(\Gamma\backslash {^\circ \!G}) $.
\end{lem}
\begin{proof}
Let $\phi\in \Hom_{^\circ \!G}(\breve F, \con^\infty(\Gamma\backslash {^\circ \!G}))$. Write 
\[
   \mathrm{Ev}:  \con^\infty(\Gamma\backslash {^\circ \!G})\rightarrow \C
\]
for the map of evaluating at $1\in \Gamma\backslash {^\circ \!G}$. Then 
\[
\mathrm{Ev}\circ \phi\in \Hom_{\Gamma}(\breve F, C)=F^\Gamma.
\]
 Lemma \ref{fixgamma0} implies that $\mathrm{Ev}\circ \phi$ is fixed by a cocompact closed subgroup $\Gamma'$ of ${^\circ \!G}$. Thus the image of $\phi$ is contained in the space of the bounded functions, and the lemma easily follows.

\end{proof}

By using the inner product induced by the invariant measure,  we know that the representation $\CA^2(\Gamma\backslash {^\circ \!G}) $ of ${^\circ \!G}$ is completely reducible. 
By Lemma \ref{auto2}, we have that
\begin{eqnarray*}
   &&(F\otimes  \CA^2(\Gamma\backslash {^\circ \!G}))^{^\circ \!G}\\
   &=&\Hom_{^\circ \!G}(\breve F, \CA^2(\Gamma\backslash {^\circ \!G}))\\
   &=&\Hom_{^\circ \!G}(\breve F, \con^\infty(\Gamma\backslash {^\circ \!G}))\\
   &=&\Hom_{\Gamma}(\breve F, \C)\\
   &=&F^\Gamma.
\end{eqnarray*}

\subsection{Proof of Theorem \ref{main0}}

Recall the linear map \eqref{maph} of the Introduction: 
\be\label{maphf}
  F^{\Gamma}\otimes \oH_{\mathrm{ct}}^i( {^\circ \!G} ; \C)= \oH_{\mathrm{ct}}^i( {^\circ \!G} ; F^{\Gamma})\rightarrow  \oH^i(\Gamma;F), \qquad(i\in \Z).
  \ee
Recall the identification \eqref{frgk22}. One checks that the map \eqref{maphf} is identical to the map
\[
   \oH_{\mathrm{ct}}^i( {^\circ \!G} ; F^{\Gamma})\rightarrow \oH_{\mathrm{ct}}^i({^\circ \!G}; F\otimes \CA(\Gamma\backslash {^\circ \!G}))
\]
induced by the linear map
\[
  F^{\Gamma}=(F\otimes  \CA^2(\Gamma\backslash {^\circ \!G}))^{^\circ \!G}\xrightarrow{\textrm{inclusion}} (F\otimes  \CA(\Gamma\backslash {^\circ \!G})).
\]

Since the representation $\CA^2(\Gamma\backslash {^\circ \!G}) $ of ${^\circ \!G}$ is completely reducible, the natural linear map
\[
   \oH_{\mathrm{ct}}^i( {^\circ \!G} ; F^{\Gamma})= \oH_{\mathrm{ct}}^i( {^\circ \!G} ;(F\otimes  \CA^2(\Gamma\backslash {^\circ \!G}))^{^\circ \!G} )\rightarrow \oH_{\mathrm{ct}}^i({^\circ \!G}; F\otimes \CA^2(\Gamma\backslash {^\circ \!G}))
\]
is always injective, and it is bijective if $i<r_{ {^\circ \!G}, F}=r_{G, F}$. Together with Proposition \ref{val223gg}, this proves Theorem \ref{main0}. 

\section{A result on semisimple real Lie algebras}

The notation of this section is independent of the previous ones.
Let $\g$ be a  semisimple real Lie algebra.

\subsection{A quantity $r_{\g}$}

Fix a Cartan involution
\be \label{cinv}
  \theta:\g\rightarrow \g,
  \ee
  and write
  \[
  \g=\k\oplus \s
  \]
   for the corresponding  Cartan decomposition. 
As before, we use a subscript $\C$ to indicate the complexification of a real vector space. Then we  have  the complexifications
   \[
    \theta: \g_\C\rightarrow \g_\C
  \]
 and
  \[
   \g_\C=\k_\C\oplus \s_\C. 
  \]
   
  If $\g$ is compact, namely it is isomorphic to the Lie algebra of a compact Lie group, we set $r_{\g}:=+\infty$. If $\g$ is noncompact and simple, 
    set
      \[
        r_{\g}:=\min\{\dim (\n_\q\cap \s_\C)\, : \, \q \textrm{ is a proper $\theta$-stable parabolic subalgebra of $\g_\C$}\},
\]
where $\n_\q$ denotes the nilpotent radical of $\q$. In general when $\g$ is noncompact, define
 \[
        r_{\g}:=\min\{r_{\g'}\, : \, \g' \textrm{ is a noncompact simple ideal of $\g$}\}.
\]
In view of the unitarizability Theorem of Vogan \cite{V} and Wallach \cite{Wa3}, the following result is a consequence of Vogan-Zuckerman's theory of unitary representations with nonzero cohomology (see \cite{Ku} and \cite[Theorem 8.1]{VZ}).

\begin{prp} Let $G$ be a real reductive group and define $r_G$ as in the Introduction. Then \[
r_G=r_{[\Lie(G), \Lie(G)]}.\]
\end{prp}

\subsection{A quantity  $r_{\g, \mu}$}

Let $\t$ be a Cartan subalgebra of $\k$. Write $\a_0$ for the centralizer of $\t$ in $\s$ so that $\h_0:=\t\oplus \a_0$ is a fundamental Cartan subalgebra of $\g$. 
Write $\Delta(\t_\C, \g_\C)$ for the root system of $\g_\C$ with respect to $\t_\C$, and likewise write $\Delta(\h_{0,\C}, \g_\C)$ for the root system of $\g_\C$ with respect to $\h_{0,\C}$.

\begin{dfnp}
An irreducible  finite dimensional representation of $\g$ is said to be pure if its infinitesimal character has a Harish-Chandra parameter in $\breve \t_\C\subset \breve \h_{0, \C}$. 

\end{dfnp}

\begin{lemp}
Let $F$ be an irreducible  finite dimensional representation of $\g$.  Then $F$ is pure if and only if there exists a pair $(G, \pi)$ such that 
\begin{itemize}
\item
  $G$ is a connected semsimple Lie group with finite center whose Lie algebra is identified with $\g$;
  \item
  the representation $F$ of $\g$ integrates to a representation of $G$;  
  \item
  $\pi$ is an irreducible unitarizable Casselman-Wallach representation of $G$ such that the total cohomology group is nonzero: 
\[
  \oH_{\mathrm{ct}}^*(G; F\otimes \pi)\neq \{0\}. 
  \]
  
\end{itemize}

\end{lemp}
\begin{proof}
This is implied by \cite[Theorem 5.6]{VZ}. 
\end{proof}

Let $\mu\in \breve \t_\C\subset \breve \h_{0, \C}$. 
Suppose that it is regular integral (with respect to $\g_\C$) in the sense that
\[
  \la \mu , \alpha^\vee\ra\in \Z\setminus\{0\},\quad \textrm{ for all roots $\alpha\in \Delta(\h_{0,\C}, \g_\C$)}.
\]
Here and as before, $\alpha^\vee\in \h_{0,\C}$ denotes the coroot corresponding  to the root $\alpha$.

The functional $\mu$ determines a positive system
 \be\label{mup}
    \Delta(\t_\C, \g_\C)^{\mu,+}:=\{\alpha\in \Delta(\t_\C, \g_\C)\, : \, \la \mu ,\alpha^\vee\ra >0\}.
 \ee
 Denote by $\rho_\mu \in \breve{\t}_\C$ the half sum of the roots in  $\Delta(\t_\C, \g_\C)^{\mu,+}$ (counted with multiplicities). Then $\mu-\rho_\mu$ determines a $\theta$-stable nilpotent subalgebra $\n(\mu)$ of $\g_\C$ whose weights are
 \[
   \{\alpha\in \Delta(\t_\C, \g_\C)\, : \, \la \mu-\rho_\mu, \alpha^\vee\ra > 0\}.
 \]
Set
\be\label{rgmu}
  r_{\g, \mu}:=\dim (\n(\mu)\cap \s_\C).
\ee

\begin{lemp}\label{lem54}
Let $G$ be a connected semsimple Lie group with finite center whose Lie algebra is identified with $\g$. 
Let $F$ be an irreducible  finite dimensional representation of $G$, and let $\pi$ be an irreducible unitarizable Casselman-Wallach representation of $G$.  If 
\[
  \oH_{\mathrm{ct}}^i(G; F\otimes \pi)\neq \{0\},
  \]
  then there exists $\lambda \in \breve \t_\C\subset \breve \h_{0, \C}$   which is a Harish-Chandra parameter of the infinitesimal character of $\breve F$ such that 
   \[
     i\geq r_{\g, \lambda}. 
   \]
\end{lemp}
\begin{proof}
This is implied by \cite[Theorem 5.6]{VZ} and \cite[Theorem 9.6.6]{Wa1}. 
\end{proof}

\subsection{The result}\label{Theresult}

Suppose that $\g$ is noncompact so that it has a  proper parabolic subalgebra $\p=\l\ltimes \n $, where  $\n $ is the nilpotent radical of $\p $, and $\l=\p\cap \theta(\p)$. 
Write
\be\label{decoml}
\l=\z\oplus \m\quad \textrm{and}\quad \z=\c \oplus \a,
\ee
where 
$\m:=[\l, \l] $ is semisimple, $\z $ is the center of $\l$, $\c:=\z\cap \k$ and $\a:=\z\cap \s$. 

Let $\c' $ be a Cartan subalgebra of $\k\cap \m$, and write $\a'$ for its centralizer in $\s\cap \m$. Then $\c'\oplus \a'$ is a fundamental Cartan subalgebra of $\m$, and 
\be\label{dh}
  \h :=\c \oplus \a \oplus \c' \oplus \a' 
\ee
is a Cartan subalgebra of $\g$. We have the complexifications 
\[
  \h_\C:=\c_\C\oplus \a_\C\oplus \c'_\C\oplus \a'_\C\quad\textrm{and}\quad  \breve \h_\C:=\breve \c_\C\oplus \breve \a_\C\oplus \breve \c'_\C\oplus \breve \a'_\C.
\]

Let $\lambda$ be a linear functional on $\h_\C$ which is regular integral in the sense that
\be\label{regulari}
  \la \lambda , \alpha^\vee\ra\in \Z\setminus\{0\},\quad \textrm{ for all roots $\alpha\in \Delta(\h_\C,\g_\C)$}.
\ee
Then $\lambda|_{\a_\C}\in \breve \a_\C\subset \breve \h_\C$, and 
\[
\la \lambda|_{\a_\C}, \alpha^\vee\ra\in \R
\]
 for all $\alpha\in \Delta(\h_\C, \n_\C)$. Here and as usual, $\Delta(\h_\C, \n_\C)\subset \Delta(\h_\C,\g_\C)$ denotes the roots associated to $\n_\C$.
The rest of this section is devoted to a proof of the following estimate.
\begin{prp}\label{estimate0}
Let the notation and the assumptions be as above. Assume that
\begin{itemize}
  \item $\la \lambda|_{\a_\C}, \alpha^\vee\ra\geq 0$, for all $\alpha\in \Delta(\h_\C, \n_\C)$, and
  \item $\lambda|_{\a'_\C}=0$. 
  \end{itemize}
 Then
\[
  \sharp\{\alpha\in \Delta(\h_\C, \n_\C)\, : \, \la \lambda, \alpha^\vee\ra>0\}+r_{\m , \lambda|_{\c'_\C}}\geq r_{\g}.
\]
\end{prp}

Note that \eqref{regulari}  and the  second assumption of Proposition \ref{estimate0} imply that $\lambda|_{\c'_\C}$ is regular integral with respcet to $\m_\C$. Thus the quantity $r_{\m , \lambda|_{\c'_\C}}$ is defined as in \eqref{rgmu}.

\subsection{Another quantity $r'_{\g}$}\label{sectable}
As before we assume that the semisimple Lie algebra $\g$ is noncompact. Set
\[
   r'_{\g}:=\min\{\dim \n_{\q }\, : \, \q \textrm{ is  a proper parabolic subalgebra of $\g$}\},
\]
where $\n_{\q }$ denotes the nilpotent radical of $\q $. Using the explicit data of noncompact simple real Lie algebras given in \cite[Appendix C, Sections 3 and 4]{Kn}, it is routine to calculate $r'_\g$ for all  noncompact simple real Lie algebras. 
On the other hand, the quantity $r_\g$ is calculated in \cite[Table 1]{En},  \cite[Theorem 2, Theorem 3]{Ku}, \cite[Table 8.2]{VZ} and \cite[Section 2.B]{LS2}. We summarize the results in the following Tables \ref{tableg}, \ref{tableg2} and \ref{tableg3}. Cartan labels are  used to denote the noncompact noncomplex exceptional  simple real Lie algebras, as in \cite[Appendix C, Section 4]{Kn}. 

We observe that in all cases, $r_{\g}'\geq r_{\g}$.

\begin{table}[h]
\caption{Complex simple Lie algebras}\label{tableg}
\centering 
\begin{tabular}{c c c c c } 
$\g$& \vline        & $r'_{\g}$ &  \vline &  $r_{\g}$\\
\hline
$\sl_n(\C)\ (n\geq 2)$ & \vline & $2(n-1)$ & \vline & $n-1$ \\
 \hline
$\s\o_n(\C) \ (n\geq 7)$ & \vline & $2(n-2)$ & \vline & $n-2$ \\
      \hline
     $\s\p_{2n}(\C) \ (n\geq 2)$ & \vline & $2(2n-1)$ & \vline & $2n-1$ \\
  \hline
 $\e_6(\C) $ & \vline & $32 $ & \vline & $16$ \\
  \hline
$\e_7(\C) $ & \vline & $54 $ & \vline & $27 $ \\
  \hline
$\e_8(\C) $ & \vline & $114 $ & \vline & $57 $ \\
  \hline
$\f_4(\C) $ & \vline & $30 $ & \vline & $15 $ \\
  \hline
$\g_2(\C) $ & \vline & $10$ & \vline & $5$ \\
  \hline


\end{tabular}
\label{table:nonlin} 
\end{table}

\begin{table}[h]
\caption{Noncompact noncomplex simple classical Lie algebras}\label{tableg2}
\centering 
\begin{tabular}{c c c c c } 
$\g$& \vline        & $r'_{\g}$ &  \vline &  $r_{\g}$\\
\hline
$\sl_n(\R)\ (n\geq 2)$ & \vline & $n-1$ & \vline & $n-1$ \\
 \hline

 $\s\p_{2n}(\R) \ (n\geq 3)$ & \vline & $2n-1$ & \vline & $n$ \\
\hline
$\s\o(p,q) \ (p\geq q\geq 1, \, p+q\geq 5, \,(p,q)\neq (3,3))$ & \vline & $p+q-2$ & \vline & $q$ \\
 \hline
$\s\u(p,q) \ (p\geq q\geq 1,\, p+q\geq 3,\,(p,q)\neq  (2,2))$ & \vline & $2(p+q)-3$ & \vline & $q$ \\
\hline
$\s\p(p,q) \ (p\geq q\geq 1,\, p+q\geq 3, \,(p,q)\neq (2,2))$ & \vline & $4(p+q)-5$ & \vline & $2q$ \\
   \hline

 $\sl_n(\BH)\ (n\geq 4)$ & \vline & $4(n-1)$ & \vline & $2(n-1)$ \\
\hline
  $\s\o^*(2n) \ (n=5,7,8,9,\cdots)$ & \vline & $4n-7$ & \vline & $n-1$ \\
       \hline
$\sl_3(\BH)$ & \vline & $8$ & \vline & $3$ \\
   \hline
$\s\o^*(12) $ & \vline & $15$ & \vline & $5$ \\
     \hline
 $\s\p(2,2) $ & \vline & $10$ & \vline & $4$ \\
 \hline

\end{tabular}
\label{table:nonlin} 
\end{table}

\begin{table}[h]
\caption{Noncompact noncomplex exceptional simple Lie algebras}\label{tableg3}
\centering 
\begin{tabular}{c c c c c } 
$\g$& \vline        & $r'_{\g}$ &  \vline &  $r_{\g}$\\
\hline
E I & \vline & $16$ & \vline & $11$ \\
 \hline
E II & \vline & $21$ & \vline & $8$ \\
 \hline
E III & \vline & $21$ & \vline & $6$ \\
 \hline
E IV & \vline & $16$ & \vline & $6$ \\
 \hline
E V & \vline & $ 27 $ & \vline & $15$ \\
 \hline
E VI & \vline & $33$ & \vline & $12$ \\
 \hline
E VII & \vline & $27$ & \vline & $11$ \\
 \hline
E VIII & \vline & $57$ & \vline & $29$ \\
 \hline
E IX  & \vline & $57$ & \vline & $24 $ \\
     \hline
 F I & \vline & $15$ & \vline & $8$ \\
 \hline
 F II & \vline & $15$ & \vline & $4$ \\
 \hline
G & \vline & $5$ & \vline & $3$ \\
 \hline


\end{tabular}
\label{table:nonlin} 
\end{table}

\subsection{A proof of Proposition \ref{estimate0}}

We continue with the notation and the assumptions of Proposition \ref{estimate0}. For the proof of Proposition \ref{estimate0}, we assume without loss of generality that $\g$ is  simple.

\begin{lemp}\label{half}
  One has that
\[
  \sharp\{\alpha\in \Delta(\h_\C, \n_\C)\, : \, \la \lambda, \alpha^\vee\ra>0\}\geq\left\lceil\frac{\dim \n }{2}\right\rceil.
\]
\end{lemp}
\begin{proof}
Note that the parabolic subalgebra  $\theta(\p)$ is opposite to $\p$. Thus
\[
  \Delta(\h_\C, \theta(\n_\C))= -\Delta(\h_\C, \n_\C),
\] 
and consequently
\[
  -\theta(\Delta(\h_\C, \n_\C))=\Delta(\h_\C, \n_\C),
\]
where $\theta: \breve \h_\C\rightarrow \breve \h_\C$ is the linear map induced by $\theta: \g_\C\rightarrow \g_\C$. 
This implies that 
\[
   \sharp\{\alpha\in \Delta(\h_\C, \n_\C)\, : \,   \la \lambda|_{\c_\C\oplus \c'_\C}, \alpha^\vee\ra> 0\}= \sharp\{\alpha\in \Delta(\h_\C, \n_\C)\, : \,   \la \lambda|_{\c_\C\oplus \c'_\C}, \alpha^\vee\ra<0\},
   \]
   and consequently, 
   \[
    \sharp\{\alpha\in \Delta(\h_\C, \n_\C)\, : \,   \la \lambda|_{\c_\C\oplus \c'_\C}, \alpha^\vee\ra\geq 0\}\geq \frac{\dim \n }{2}. 
   \]

Recall that $\lambda$ is regular integral, and note that 
\[
  \la \lambda|_{\a_\C\oplus \a'_\C}, \alpha^\vee\ra\geq 0, \quad \textrm{for all }\alpha\in \Delta(\h_\C, \n_\C).
\]
Therefore we have that
\[
   \sharp\{\alpha\in \Delta(\h_\C, \n_\C)\, : \, \la \lambda, \alpha^\vee\ra> 0\}=\sharp\{\alpha\in \Delta(\h_\C, \n_\C)\, : \, \la \lambda, \alpha^\vee\ra\geq 0\}\geq \frac{\dim \n }{2}.
\]
This proves  the lemma.

\end{proof}

\begin{lemp}\label{lhalf}
If $\g$ is not isomorphic to a split simple real Lie algebra of type $A_n$ ($n\geq 2$), $D_n$ ($n\geq 4$), $E_6$ or $E_7$, then
\[
  \left\lceil\frac{r_{\g}'}{2}\right\rceil \geq r_{\g}.
\]
\end{lemp}
\begin{proof}
This follows form Tables \ref{tableg}, \ref{tableg2} and \ref{tableg3}.
\end{proof}

By Lemma \ref{half} and Lemma \ref{lhalf}, Proposition \ref{estimate0} holds when $\g$ is not isomorphic to a split simple real Lie algebra of type $A_n$ ($n\geq 2$), $D_n$ ($n\geq 4$), $E_6$ or $E_7$.

We will need the following general result.
\begin{lemp}\label{lhalfrho}
Let $\l'_\C$ be a reductive finite dimensional complex Lie algebra, and let $F'$ be an irreducible finite dimensional representation of $\l'_\C$. Let $\h'_\C$ be  a Cartan subalgebra of $\l'_\C$ and fix a positive system $\Delta(\h'_\C, \l'_\C)^+$ of the root system $\Delta(\h'_\C, \l'_\C)$. Write $\rho'^\vee\in \h'_\C$ for the half sum of the positive coroots associated to $\Delta(\h'_\C, \l'_\C)^+$. Then 
the set of eigenvalues of $\rho'^\vee$ on $F'$ has the form 
\[
   b_0+\{j\in \Z\, :\, 0\leq j\leq c_0\},
\]
where $b_0\in \frac{1}{2} \Z$, and $c_0\geq 0$ is an  integers. 
\end{lemp}
\begin{proof}
Let $\alpha_1, \alpha_2, \cdots, \alpha_k$ ($k\geq 0$) be the simple roots in $\Delta(\h'_\C, \l'_\C)^+$.
 Let $X_j$ be a root vector attached to $\alpha_j$ ($1\leq j\leq k$). Let $v_0$ be a lowest weight vector in $F'$. Then
\[
\rho'^\vee.v_0=b_0 v_0
\]
for some $b_0\in \frac{1}{2} \Z$. Note that $F'$ is spanned by vectors of the form 
\[
  (X_{j_1} X_{j_2} \cdot \ldots \cdot X_{j_r}).v_0, 
\]
where $r\geq 0$, $1\leq j_1, j_2, \ldots, j_r\leq k$. The lemma then easily follows by noting that
\[
   \rho'^\vee.((X_{j_1} X_{j_2} \cdot \ldots \cdot X_{j_r}).v_0)=(b_0+r)((X_{j_1} X_{j_2} \cdot \ldots \cdot X_{j_r}).v_0).
\]  

\end{proof}

Let $\Check{\g}_\C$ denote the Langlands dual Lie algebra of $\g_\C$, with Cartan subalgebra $\check \h_\C$ equals the dual of $\h_\C$.
Let $\check \p_\C=\check \l_\C\ltimes \check \n_\C$ denote the parabolic subalgebras of $\check \g_\C$ corresponding to $\p_\C$, where $\check \l_\C\supset \check \h_\C$ is the Levi factor corresponding to $\l_\C$,  and $\check \n_\C$ is the nilpotent radical corresponding to $\n_\C$. Then $\check \m_\C:=[\check \l_\C, \check \l_\C]$ is the Langlands dual of $\m_\C$. 

\begin{lemp}\label{half04}
Assume that $\m $ is nonzero and contains no compact simple ideal, and that the adjoint representation of $\check \l_\C$ on $\check \n_\C$ is irreducible. Then
\[
 \sharp\{\alpha\in \Delta(\h_\C, \n_\C)\, : \, \la \lambda, \alpha^\vee\ra> 0\}=\dim \n ,
\]
or
\[
  r_{\m , \lambda|_{\c'_\C}} \geq r_{\m }.
\]
\end{lemp}
\begin{proof}
For simplicity write $\lambda':=\lambda|_{\c'_\C}$. It is regular integral with respect to $\m_\C$. As in  \eqref{mup} we define the positive root system $\Delta(\c'_\C, \m_\C)^{\lambda',+}$, and write 
 $\rho_{\lambda'} \in \breve{\c'}_\C$ for the half sum of the positive roots (counted with multiplicities).

If $\lambda'\neq \rho_{\lambda'}$, then the assumption on $\m$ implies that 
\[
  r_{\m , \lambda'} \geq r_{\m }.
\]
So we assume that
$\lambda'= \rho_{\lambda'}$. 

Write
\[
  \lambda= \lambda|_{\c_\C\oplus \a_\C}+ \lambda|_{\c'_\C\oplus \a'_\C}.
\]
Note that $\check \h_\C=\breve \h_\C$, and $\breve \c_\C\oplus \breve \a_\C$ is identified with the center of $\check \l_\C$. By Schur's lemma, the irreducibility assumption implies that 
$\lambda|_{\c_\C\oplus \a_\C}\in \breve \c_\C\oplus \breve \a_\C\subset \check \l_\C$ acts on $\check \n_\C$ by the multiplication of a constant $a_0\in \C$. Thus
\be\label{a0}
   \la \lambda|_{\c_\C\oplus \a_\C}, \alpha^\vee\ra=a_0\qquad \textrm{for all }\alpha\in \Delta(\h_\C, \n_\C).
\ee

Note that \[
\rho_{\lambda'}\in \breve \c'_\C\subset  \breve \h_\C
\]
is also the half sum of a positive system of the root system $\Delta(\h_\C, \l_\C)$. We identify it with an element of $\check \h_\C$ which is the half sum of the coroots associated to a 
positive system of the root system $\Delta(\check \h_\C, \check \l_\C)$.  Then by Lemma \ref{lhalfrho}, 
\be\label{b0}
  \{\la \rho_{\lambda'}, \alpha^\vee\ra \, : \, \alpha\in \Delta(\h_\C, \n_\C) \}=b_0+\{j\in \Z \, : \, 0\leq j\leq c_0\}, 
\ee
where $b_0\in \frac{1}{2} \Z$, and $c_0\geq 0$ is an  integers. 

Since $\lambda|_{\c'_\C}= \rho_{\lambda'}$ and $\lambda|_{\a'_\C}=0$, by combining \eqref{a0} and \eqref{b0}, we get that
\[
   \{\la \lambda, \alpha^\vee\ra \, : \, \alpha\in \Delta(\h_\C, \n_\C) \}=a_0+b_0+\{j\in \Z \, : \, 0\leq j\leq c_0\}.
\]
Because $\lambda$ is regular integral,  the above set either consists of only positive integers, or consists only negative integers. By  Lemma \ref{half}, it has to consists only of positive integers. 
This proves the lemma.
\end{proof}

\begin{lemp}\label{lhalf6}
(a).  Assume that $\g=\sl_n(\R)$ ($n\geq 3$). If $\p $ is not a maximal parabolic subalgebra, then
\[
  \left\lceil\frac{\dim \n }{2}\right\rceil\geq r_{\g}.
\]
If $\p $ is a maximal parabolic subalgebra, then
\[
  \left\lceil\frac{\dim \n }{2}\right\rceil+r_{\m }\geq r_{\g}.
\]

(b) Assume that $\g=\s\o(n,n)$ ($n\geq 4$). If $\m $ is not isomorphic to $\s\o(n-1,n-1)$ or $\sl_{n}(\R)$, then
\[
  \left\lceil\frac{\dim \n }{2}\right\rceil\geq r_{\g}.
\]
If $\m $ is isomorphic to $\s\o(n-1,n-1)$ or $\sl_{n-1}(\R)$, then
\[
  \left\lceil\frac{\dim \n }{2}\right\rceil+r_{\m }\geq r_{\g}.
\]

(c) Assume that $\g$ is split of type $E_6$. If $\m $ is not isomorphic to $\s\o(5,5)$, then
\[
  \left\lceil\frac{\dim \n }{2}\right\rceil\geq r_{\g}.
\]
If $\m $ is isomorphic to $\s\o(5,5)$, then
\[
  \left\lceil\frac{\dim \n }{2}\right\rceil+r_{\m }\geq r_{\g}.
\]

(d) Assume that $\g$ is split of type $E_7$. If $\m $ is not isomorphic to the split simple Lie algebra of type $E_6$, then
\[
  \left\lceil\frac{\dim \n }{2}\right\rceil\geq r_{\g}.
\]
If $\m $ is isomorphic to the split simple Lie algebra of type $E_6$, then
\[
  \left\lceil\frac{\dim \n }{2}\right\rceil+r_{\m }\geq r_{\g}.
\]

\end{lemp}
\begin{proof}
This is routine to check.
\end{proof}

\begin{lemp}\label{half44}
Proposition \ref{estimate0}  holds when $\g$ is isomorphic to $\sl_n(\R)$ ($n\geq 3$).
\end{lemp}
\begin{proof}
If $\p$ is not a maximal parabolic subalgebra, then the lemma follows form Lemma \ref{half} and the first assertion of part (a) of Lemma \ref{lhalf6}.
If $\p$ is a maximal parabolic subalgebra, then the representation of $\check \l_\C$ on $\check \n_\C$ is irreducible. The lemma then follows from Lemma  \ref{half}, Lemma \ref{half04}, and the second assertion of part (a) of Lemma \ref{lhalf6}.

\end{proof}

The same proof as  Lemma \ref{half44} shows that Proposition \ref{estimate0} also holds when $\g$ is isomorphic to a split simple real Lie algebra of type $D_n$ ($n\geq 4$), $E_6$ or $E_7$. This completes the proof of Proposition \ref{estimate0}.

\section{A proof of Theorem \ref{mainlocal}}

Let the notation be as in Section \ref{s1}.  To prove Theorem \ref{mainlocal}, we assume without loss of generality that $G$ is connected and semisimple. 
As in Theorem \ref{mainlocal}, $P$ is a proper parabolic subgroup of $G$, $N$ is its unipotent radical, and $L:=P/N$.

Shapiro's lemma implies that 
\[
  \oH_{\mathrm{ct}}^{i}(G;  F\otimes \Ind_P^G (\nu\otimes\sigma))=\oH_{\mathrm{ct}}^{i}(P;  F\otimes \rho_P \otimes \nu\otimes\sigma),  \quad (i\in \Z),
\]
where $\rho_P$ is the positive character of $P$ whose square equals the modular character of $P$.
We have a spectral sequence
\[
   E_2^{i, j}:=\oH_{\mathrm{ct}}^i(L; \oH^j(\n_\C, F)\otimes \rho_P\otimes (\nu\otimes \sigma))\Rightarrow \oH_{\mathrm{ct}}^{i+j}(P;  F\otimes \rho_P \otimes \nu\otimes\sigma).
   \]

Write $L_0$ for the  identity connected component of $L$. Note that $\oH^j(\n_\C; F)$ is completely reducible as a representation of $L_0$. 
Assume that 
\[
\oH_{\mathrm{ct}}^i(L; \oH^j(\n_\C; F)\otimes \rho_P\otimes (\nu\otimes \sigma))\neq 0,
\]
for some $i,j\in \Z$. Then 
\be\label{nonvhi}
\oH_{\mathrm{ct}}^i(L_0; F_0\otimes \rho_P\otimes (\nu\otimes \sigma_0))\neq 0,
\ee
for some unitarizable irreducible Casselman-Wallach representations $\sigma_0$ of  $L_0$, and some irreducible $L_0$-subrepresentation $F_0$ of $ \oH^j(\n_\C; F)$.

Fix a Cartan involution $\theta: G\rightarrow G$, and still denote its differential by $\theta : \g\rightarrow \g$. Identify $L$ with $P\cap \theta(P)$. Then its Lie algebra $\l$ has a decomposition
\[
  \l=\c\oplus \a\oplus \m
\]
as in \eqref{decoml}. 
Let 
\be\label{dh22}
  \h :=\c \oplus \a \oplus \c' \oplus \a' 
\ee
be a Cartan subalgebra of $\l$ which is obtained as in \eqref{dh}.

By Lemma \ref{lem54}, there exists  $\lambda\in \breve \h_\C$ such that 
\begin{itemize}
   \item $-\lambda$ is  a Harish-Chandra parameter of the infinitesimal character of the representation $F_0\otimes \rho_P$ of $L_0$;
  \item $\lambda|_{\a'_\C}=0$; and 
  \item $i\geq r_{\m , \lambda|_{\c'_\C}}$.
\end{itemize}
The first of the above three conditions in particular implies that
\[
  \lambda|_{\a_\C}=\nu|_{\a_\C}.
\] 
 By  Casselman-Osborne Theorem (see \cite[Theorem 4.149]{KV}), $-\lambda$ is also a Harish-Chandra parameter of the infinitesimal character of the representation $F$ of $G$. Hence $\lambda$ is regular integral. 
 
\begin{lemp}
The assumption that $F_0$ is an irreducible $L_0$-subrepresentation of $\oH^j(\n_\C; F)$  implies that
\[
  j=\sharp\{\alpha\in \Delta(\h_\C, \n_\C)\, : \, \la \lambda, \alpha^\vee\ra>0\}.
\]

\end{lemp}
\begin{proof}
This follows from Kostant's Theorem (see \cite[Theorem 4.139]{KV}). 
\end{proof} 
 
 Now  Proposition \ref{estimate0} implies that
 \[
   i+j\geq \sharp\{\alpha\in \Delta(\h_\C, \n_\C)\, : \, \la \lambda, \alpha^\vee\ra>0\}+r_{\m , \lambda|_{\t'_\C}}\geq r_{\g}. 
 \]
 This proves Theorem \ref{mainlocal}.

\section*{Acknowledgements}

Jian-Shu Li's research was partially supported by RGC-GRF grant 16303314 of HKSAR.

Binyong  Sun was supported in part by the National Natural Science Foundation of China (No.  11525105, 11688101, 11621061 and 11531008).

\end{document}